\documentclass[10pt,a4paper]{amsart}
\usepackage[toc,page]{appendix}
\usepackage{a4}
\usepackage[active]{srcltx}
\usepackage{amsmath,bbm,esint}
\usepackage{amssymb}
\usepackage{amsfonts}
\usepackage{mathrsfs} 
\usepackage{graphicx}
\usepackage{tikz}
\usepackage{color}
\usepackage[colorlinks=true,urlcolor=blue,
citecolor=red,linkcolor=blue,linktocpage,pdfpagelabels,bookmarksnumbered,bookmarksopen]{hyperref}

\newcounter{hyp}
\def\a{\alpha}
\def\b{\beta}

\def\l{\lambda}

\def\vphi{\varphi}

\def\t{\tau}
\def\e{\varepsilon}

\newcommand{\cC}{{\mathcal C}}

\newcommand{\cF}{{\mathcal F}}
\newcommand{\cG}{{\mathcal G}}
\newcommand{\cL}{{\mathcal L}}

\newcommand{\cP}{{\mathcal P}}
\newcommand{\cS}{{\mathcal S}}

\newcommand{\R}{{\mathbb R}}

\newcommand{\N}{\mathbb{N}}
\newcommand{\M}{{\mathbb R^d}}

\newcommand{\ov}{\overline}

\newcommand{\mres}{\mathbin{\vrule height 1.6ex depth 0pt width
0.13ex\vrule height 0.13ex depth 0pt width 1.3ex}}

\newcommand{\sD}{\mathscr{D}}

\newcommand{\sP}{{\mathscr P}}

\newcommand{\ds}{\displaystyle}

\newcommand{\dd}{\hspace{0.7pt}{\rm d}}

\newcommand{\be}{\begin{equation}}
\newcommand{\ee}{\end{equation}}

\newcommand{\ba}{\begin{array}}
\newcommand{\ea}{\end{array}}

\newtheorem{theorem}{Theorem}[section]
\newtheorem{definition}[theorem]{Definition}
\newtheorem{lemma}[theorem]{Lemma}
\newtheorem{corollary}[theorem]{Corollary}

\newtheorem{remark}[theorem]{Remark}

\newtheorem{innercustomgeneric}{\customgenericname}
\providecommand{\customgenericname}{}
\newcommand{\newcustomtheorem}[2]{%
  \newenvironment{#1}[1]
  {%
   \renewcommand\customgenericname{#2}%
   \renewcommand\theinnercustomgeneric{##1}%
   \innercustomgeneric
  }
  {\endinnercustomgeneric}
}

\newcustomtheorem{customthm}{Framework}

\numberwithin{equation}{section}

\title[]{Mean Field Games systems under displacement monotonicity}

\author[A.R. M\'esz\'aros]{Alp\'ar R. M\'esz\'aros}  
\date{\today}
\address{Department of Mathematical Sciences, University of Durham, Durham DH1 3LE, England}
\email{alpar.r.meszaros@durham.ac.uk} 

\author[C. Mou]{Chenchen Mou}
\address{Department of Mathematics, City University of Hong Kong, Hong Kong SAR, China}
\email{chencmou@cityu.edu.hk}
\thanks{
Data Availability: Data sharing not applicable to this article as no datasets were generated or analyzed during the current study.
}

\begin{document}
\maketitle

\begin{abstract} 
In this note we prove the uniqueness of solutions to a class of Mean Field Games systems subject to possibly degenerate individual noise. Our results hold true for arbitrary long time horizons and for general non-separable Hamiltonians that satisfy a so-called \emph{displacement monotonicity} condition. This monotonicity condition that we propose for non-separable Hamiltonians is sharper and more general than the one proposed in the work \cite{GanMesMouZha}. The displacement monotonicity assumptions imposed on the data provide actually not only uniqueness, but also the existence and regularity of the solutions. Our analysis uses elementary arguments and does not rely on the well-posedness of the corresponding master equations.
\end{abstract}


\section{Introduction}

The theory of Mean Field Games (MFG for short in the sequel) was introduced around 2006 simultaneously by Lasry-Lions (\cite{LL06a, LL06b, LL07a, Lions}) and  Huang-Malham\'e-Caines (\cite{HCM06,HCM071,HCM072,HCM073}). Since then, its literature has witnessed a vast increase in various directions and the theory turned out to be extremely rich in applications.

In its simplest form (cf. \cite{Cardaliaguet,CarDelLasLio,CarPor,CarDel-I,CarDel-II}), an MFG can be fully characterized by the solutions of the following system of nonlinear PDEs.
\begin{equation}\label{eq:mfg}
\left\{
\begin{array}{ll}
-\partial_t u(t,x) -\frac{\beta^2}{2}\Delta_xu(t,x)+ H(x,-D_xu(t,x),\rho_t)=0,&  {\rm{in}}\ (0,T)\times\M,\\[3pt]
\partial_t\rho_t-\frac{\beta^2}{2}\Delta_x\rho_t+D_x\cdot(\rho_t D_pH(x,-D_x u(t,x),\rho_t))=0, & {\rm{in}}\ (0,T)\times\M,\\[3pt]
u(T,x)=g(x,\rho_T),\ \rho(0,\cdot)=\rho_0, & {\rm{in}}\ \M,
\end{array}
\right.
\end{equation}
where $\beta\in\R$ is the intensity of the individual noise, $T>0$ is a given time horizon, and $\rho_0\in\sP_2(\R^d)$ is the initial configuration of the agents. Here, the state space of the agents is represented by $\R^d$ and $\sP_2(\R^d)$ denotes the space of Borel probability measures on $\R^d$ with finite second moments. 

We underline that the unknown $u:[0,T]\times\R^d\to\R$ stands for the value function of a typical agent, who solves the control problem
\begin{align*}
u(t,x)=\inf_\alpha\mathbb{E}\left\{\int_t^T L(X_s,\alpha_s,\rho_s)\dd s+ g(X_T,\rho_T)\right\},\ \ {\rm{s.\ t.}}\ \ 
\left\{
\begin{array}{ll}
\dd X_s=\alpha_s\dd s +\beta\dd B_s, & s\in(t,T),\\
X_t = x,
\end{array}
\right.
\end{align*}
where $g:\R^d\times\sP(\R^d)\to\R$ is a given final cost function and $L:\R^d\times\R^d\times\sP(\R^d)\to\R$ is a given Lagrangian function  that models the running cost. The unknown $\rho:[0,T]\to\sP(\R^d)$, the distribution of the agent population, enters into this optimization problem. The Hamiltonian $H:\R^d\times\R^d\times\sP(\R^d)\to\R$ is simply defined as $H(x,\cdot,\mu)=L^*(x,\cdot,\mu)$ for all $x\in\M$ and $\mu\in\sP(\R^d)$, i.e. it is the Legendre--Fenchel transform of $L$, in its second variable. When $\beta= 0$, the model becomes deterministic.

By now, the well-posedness of system \eqref{eq:mfg} is well understood in many different settings and the first results date back to the original works of Lasry and Lions and have been presented in the course of Lions at Coll\`ege de France (cf. \cite{Lions}). A complete account on the progress of the literature on this subject has been recently published in the self contained and well-written lecture notes \cite{CarPor} from the PDE viewpoint and in the monographs \cite{CarDel-I,CarDel-II} from the probabilistic viewpoint. Let us now discuss the state of the art of the literature, that will be relevant for our considerations.

\medskip

\noindent {\bf Literature overview.} Regarding $H$ and $g$, we can consider nonlocal (regularizing) and local dependence on the measure variable $\mu$. If $\beta\neq 0$, system \eqref{eq:mfg} possesses a parabolic structure. When $H$ and $g$ are nonlocal and regularizing in the measure variable, it is fairly straightforward to obtain existence of a classical solution under very general assumptions on $H$ and $g$ for any $\rho_0\in\sP(\R^d)$ and $T>0$ (cf. \cite{CarPor, CarDel-I}). 

If $H$ and $g$ are local functions of the density variable, for general Hamiltonians the well-posedness result (classical or weak solutions) is known only for short time (cf. \cite{Amb:18, Amb:21, CirGiaMan}). For arbitrary long time horizon $T>0$, the existence of (classical or weak) solutions  is known only under additional structural assumptions on $H$. This is for instance, when $H$ possesses a so-called {\it separable} structure (cf. \cite{CarPor,CirGof,GomPimSan:15, GomPimSan:16,Por}, i.e. the momentum and measure variables are additively separated, having the form of 
\begin{equation}\label{sep}
H(x,p,\mu)=H_0(x,p)-f(x,\mu)
\end{equation} 
for some $H_0$ and $f$.

When $\beta=0$ and $T>0$ is arbitrary,  the existence of a weak solution to \eqref{eq:mfg} is known only under the condition \eqref{sep} and with extra assumptions on the initial measure $\rho_0$ (such as boundedness or compact support, see \cite{Cardaliaguet,CarPor} in the case of nonlocal regularizing data $f,g$; and \cite{Car, CarGra,CarGraPorTon, CarPor} in the case of locally depending $f,g$ on the $\mu$ variable). 

\medskip

Now, let us turn our attention to the question of uniqueness of solutions to \eqref{eq:mfg}. As expected, this is a more subtle question and additional assumptions must be imposed to hope for positive results in this direction. Already in their original works, when \eqref{sep} takes place, Lasry and Lions proposed a notion of monotonicity (which bears the name of {\it Lasry--Lions monotonicity} in the literature now) on the coupling functions $f$ and $g$ under which uniqueness of solutions to \eqref{eq:mfg} can be obtained, as long as they are regular enough. Indeed, let us underline, that for instance when $\beta= 0$ and $f$ and $g$ are nonlocal and regularizing, the Lasry--Lions monotonicity implies the uniqueness of solutions as long as the measure component $\mu$ of the solution is essentially bounded (cf. \cite[Theorem 1.8]{CarPor}). When $\beta\neq0$ the parabolic regularity kicks in, and so the uniqueness result holds under the Lasry--Lions monotonicity condition, without any additional assumption (cf. \cite[Theorem 1.4]{CarPor}). When $f$ and $g$ are local functions of the density variable, the existence and (partial) uniqueness of weak solutions can be obtained by variational techniques as in \cite{Car,CarGra, CarGraPorTon}.

\medskip

Relying on the examples of non-uniqueness of solutions in the lack of the Lasry--Lions monotonicity, provided in \cite{BarFis, BayZha, BriCar,CarDel-I, Lions}, one might wonder whether the Lasry--Lions monotonicity is a necessary condition for the uniqueness of solutions. 
When $H$ and $g$ are nonlocal regularizing functions in the measure variable, until recently the global in time uniqueness of solutions to \eqref{eq:mfg} was essentially known only in the regime of separable Hamiltonians that satisfy the Lasry--Lions monotonicity condition. In this paper our goal is to present a different regime which can take care of a class of data outside of the scope of the Lasry--Lions monotonicity.

\medskip

When $H$ and $g$ are local functions of the density variable, Lions in his lectures (cf. \cite{Lions}) provided a general monotonicity condition on $H$ which yields the uniqueness of solutions (see also \cite{AchPor, LioSou}, where this condition has been exploited). Finally, recently a general framework based on monotone operators in Banach spaces (cf. \cite{FerGom, FerGomTad}) has been proposed to show the well-posedness of general MFG systems. These all can be seen as generalizations of the Lasry--Lions monotonicity condition in the case of  possibly non-separable, but special Hamiltonians, depending locally on the density variable. We underline that to the best of our knowledge, no such generalization of the Lasry--Lions monotonicity is known in the case of non-separable Hamiltonians that are nonlocal in the measure variable. The Lasry--Lions monotonicity condition is certainly a sufficient one, which in many cases provides the well-posedness (hence uniqueness) of MFG systems. In some cases it can be even used to obtain higher order regularity of weak solutions to first order local systems (cf. \cite{GraMes,GraMesSilTon}) and stability and convergence of numerical schemes (\cite{AlmFerGom, GomSau:21}). 

\medskip

The uniqueness and stability of solutions to \eqref{eq:mfg} plays an instrumental role in the theory. For instance, the well-posedness of the associated {\it master equations} -- introduced by Lions -- is known so far only under the uniqueness and stability of solutions to the MFG system (cf. \cite{CarDelLasLio, CarDel-II, ChaCriDel, MouZha}). On contrary, the well-posedness of the master equation also implies uniqueness of the associated MFG system.

\medskip

The recent results \cite{ GanMesMouZha} on the well-posedness of the master equations in the presence of individual and common noises in a different regime of monotonicity (the so-called {\it displacement monotonicity}) suggests that there are conditions other than the Lasry--Lions monotonicity that could lead to the global in time well-posedness of master equations. 
As it is detailed in \cite{GanMesMouZha}, the displacement monotonicity condition is in general in dichotomy with the Lasry--Lions monotonicity, and it allows to treat a general class of non-separable Hamiltonians. 
We note that the displacement monotonicity stems from the notion of displacement convexity arising in optimal transport theory (cf. \cite{McCann1997}), which has been already used to study potential MFG in the deterministic case (cf. \cite{BenGraYam,GanMes}) and in the stochastic case (cf. \cite{ChaCriDel}). It seems that \cite{Ahuja} (see also \cite{CarDel:15,AhujaRenYang,ChaCriDel}), whose weak monotonicity condition is essentially equivalent to the displacement monotonicity (in the case of particular separable Hamiltonians), is the first work that relied on displacement monotonicity to study the well-posedness of McKean--Vlasov FBSDEs with a special form. Interestingly, the nonlocal coupling function considered in \cite{Bar} has a displacement monotone structure. So we believe that our techniques might lead to a better understanding of the uniqueness issues raised there.

\medskip

Thus, the purpose of this manuscript is to present the well-posedness of the MFG system \eqref{eq:mfg} in the case of a general class of non-separable Hamiltonians and final cost functions that possess the appropriate displacement monotonicity assumptions, for arbitrary time horizons $T$ and possibly degenerate individual noise. We emphasize that in this note we are using only elementary analysis. This means that we do not rely on the well-posedness of the corresponding master equations. In particular, we are using only some classical tools from stochastic control theory, the theory of viscosity solutions, and Fokker--Planck type equations. 

%

\medskip

\noindent {\bf Our main results.} The heart of our analysis lies in the fact that the displacement monotonicity assumption on the data (which in fact implies convexity of $x\mapsto g(x,\mu)$ and $(x,v)\mapsto L(x,v,\mu)$ for all $\mu\in\sP_2(\mathbb R^d)$) together with classical results from optimal control theory imply that the solution $u(t,\cdot)$ of the HJB equation from \eqref{eq:mfg} has a $C^{1,1}$ a priori estimate in the space variable $x$, independently of the intensity $\beta$ of the noise. First, having this regularity in hand yields the existence of a solution to the system \eqref{eq:mfg} when $\beta=0$. Indeed, in the deterministic case (i.e. $\beta= 0$) the solution to the continuity equation can be represented via the flow of a Lipschitz continuous vector field, so we can get enough compactness to formulate a fixed point problem, which in turn yields the existence result. Let us remark that in the lack of such a priori estimate on $u$ (that would have only semi-concavity estimates), a more sophisticated argument is needed (by passing though the DiPerna--Lions theory) to obtain weak solutions to the continuity equation (as explained in \cite{CarPor,CarSou}), and so, additional assumptions on the structure of the Hamiltonian and the initial measure $\rho_0$ seem to be necessary. Our existence results in the case of deterministic problems seem to be new in the literature (as we can consider general initial measures $\rho_0\in\sP_2(\R^d)$). The philosophy behind our results is the same also in the parabolic setting, when $\b\neq0$. We state in an informal way here one of our main results and will give the full details on it in Theorem \ref{thm:exist_deter}.

\begin{theorem}
Assume that the Lagrangian function $L$ and the final cost function $g$ satisfy certain regularity conditions and growth conditions. Assume further that the functions $x\mapsto g(x,\mu)$ and $(x,v)\mapsto L(x,v,\mu)$ are convex. Then the mean field game system \eqref{eq:mfg}, with $\beta=0$, has a solution pair $(u,\rho)$. 
\end{theorem}

\medskip

The a priori $C^{1,1}$ regularity on $u(t,\cdot)$ justifies the space of the solutions we consider for the uniqueness. Furthermore, this has another deep consequence: together with the displacement monotonicity of the data, this implies 
a sort of monotonicity property for the difference $D_xu^1-D_xu^2$ along any two solutions $(u^1,\rho^1)$ and $(u^2,\rho^2)$ to the systems \eqref{eq:mfg} with initial distributions $\rho^1_0$ and $\rho^2_0$, respectively. 
In fact this result implies the propagation of the displacement monotonicity of the solution to the corresponding master equation. This is the crucial property that yields a Gr\"onwall type estimate on $W_2(\rho^1_t,\rho^2_t)$ from where the uniqueness follows. More precisely, we establish the following stability result, which is stated informally here. See Theorem \ref{thm:unique-decay} for the precise result.

\begin{theorem}
Assume that the Lagrangian function $L$ and the final cost function $g$ satisfy certain regularity conditions and growth conditions. Assume further that the functions $L$ and $g$ satisfy the displacement monotonicity condition. Let $(u^1,\rho^1)$ and $(u^2,\rho^2)$ be two solution pairs to \eqref{eq:mfg} with initial data $\rho^1_0,\rho^2_0\in\sP_2(\R^d)$, respectively. Then there exists $C>0$ depending only on $T$ and the data such that
\begin{align*}
\sup_{t\in[0,T]}W_2(\rho^1_t,\rho^2_t)\le C W_2(\rho^1_0,\rho^2_0),
\end{align*}
and
\begin{align*}
\sup_{t\in[0,T]}\left\|D_xu^1(t,\cdot)-D_xu^2(t,\cdot)\right\|_{L^\infty(\mathbb R^d)}\le CW_2(\rho_0^1,\rho_0^2).
\end{align*}
\end{theorem}

\medskip
\noindent {\bf Comparison with earlier results involving displacement monotone data.}
As mentioned above, displacement monotonicity (although under a different name) was used for the first time in the context of mean field games in the work \cite{Ahuja} (see also \cite{CarDel:15,AhujaRenYang}) to show uniqueness of solutions to MFG with common noise. These works involved separable Hamiltonians. Later, in the context of the master equations displacement convexity and monotonicity (in the case of separable Hamiltonians) was used in the works \cite{BenGraYam,ChaCriDel,GanMes}. Displacement monotonicity to study mean field games master equations in the case of non-separable Hamiltonians was proposed in the recent work \cite{GanMesMouZha}. Although the well-posedness of the master equation implies the uniqueness of solutions to the corresponding MFG system, the standing assumptions in this manuscript differ significantly from the ones in \cite{GanMesMouZha}. Firstly, \cite{GanMesMouZha} imposed the presence of a non-degenerate noise (i.e. $\beta\neq0$), and the current manuscript is able to handle degenerate, deterministic problems. Secondly, \cite{GanMesMouZha} assumed that the data have high order derivatives (and various bounds on those). In particular, there for instance the final condition $g$ cannot have quadratic growth at infinity in the $x$-variable. In the current manuscript we impose merely $C^{1,1}$ type regularity assumptions on the data, that can have more general growth condition at infinity. Last, but most importantly, the displacement monotonicity assumption on non-separable Hamiltonians that we propose in this manuscript improves the corresponding condition from \cite{GanMesMouZha}. We show that the condition from \cite{GanMesMouZha} always implies our newly proposed condition, but these are in general not equivalent. 

Finally we would like mention that there are some other works on mean field games and planning problems which also use techniques relying on displacement convexity (not the displacement monotonicity). In \cite{BakFerGom,GomSen,LavSan:15} the authors identified functionals on probability measures which are convex along the measure flow component of first order MFG systems (or along discrete in time iterations of such) and planning problems. This information is then used to obtain new a priori estimates. In certain cases, these in particular could lead to $L^\infty$ estimates for the density of the distribution of the agents, in the case of deterministic problems.

\medskip

We expect that the techniques developed in this manuscript could be applied to study various other problems in the regime of displacement monotone data, such as the long time behavior of both MFG systems and master equations (cf. \cite{CarPor:19,CarPor, CirPor}), weak solutions for the master equation (cf. \cite{MouZha}), classical solutions to degenerate master equations subject to common noise (cf. \cite{CarSou}) and others. We believe that our newly proposed displacement monotonicity condition for non-separable Hamiltonians could serve as a sharper condition for the well-posedness of the corresponding master equation. We leave such investigations to future works.

\medskip

The structure of the rest of the paper is given as follows. In Section \ref{sec:assumptions} we have collected all the assumptions on our data. Section \ref{sec:existence} presents the existence of a solution to the system \eqref{eq:mfg}. We end the note with Section \ref{sec:uniqueness} that contains the main results on the uniqueness of the solutions.

\section{Standing assumptions}\label{sec:assumptions}

Throughout the note, let $T>0$ be any given arbitrary time horizon, and $(\Omega,\mathbb F,\mathbb P)$ be a filtered probability spaces, on which is defined a standard $d$-dimensional Brownian motion $B$. For $\mathbb F=\{\mathcal{F}_t\}_{0\leq t\leq T}$, we assume $\mathcal{F}_t=\mathcal{F}_0\vee\mathcal{F}^{B}$, and $\mathbb P$ has no atom in $\mathcal{F}_0$, so it can support any measure on $\mathbb R^d$ with finite second moment, i.e. the map $X\mapsto \mathbb{P}\circ X^{-1}$ is surjective from $\mathbb{L}^2(\cF_T)\to\sP_2(\R^d)$.

Let us introduce now the Wasserstein space and the differential calculus on it. For any $p\geq 1$, let $\sP_p(\mathbb R^d)$ stand for the set of Borel probability measures with finite $p$-moment and for $\mu\in\sP_p(\R^d)$ we denote its $p$-moment by $M_p(\mu):=\left(\int_{\R^d}|x|^p\dd\mu(x)\right)^{\frac1p}$. For any sub-$\sigma$-field $\cG\subset \mathcal{F}_T$ and $\mu\in\sP_p(\mathbb R^d)$, denote by $\mathbb L^p(\cG)$ the set of $\mathbb R^d$-valued, $\cG$-measurable, and $p$-integrable random variables $\xi$; and by $\mathbb L^{p}(\cG;\mu)$ the set of $\xi\in\mathbb L^p(\cG)$ such that the law $\cL_{\xi}:=\xi_\sharp\mathbb{P}=\mu$. Here, by $\xi_\sharp\mathbb{P}=\mu$ we denoted the push-forward of $\mathbb{P}$ by $\xi$, i.e. $\mu(A)=\mathbb P\circ \xi^{-1}(A)$ for any Borel set $A\subseteq\mathbb R^d$. If $A\subseteq \R^d$ is a Borel set, by $\mu\mres A$ we denote the restriction of $\mu$ to $A$, i.e. $(\mu\mres A)(D):=\mu(A\cap D)$ for any Borel set $D\subseteq\R^d$.

For any $\mu,\nu\in\sP(\mathbb R^d)$, their $W_p$-Wasserstein distance is defined as
\begin{equation}\label{eq:Wass}
W_p(\mu,\nu):=\inf\left\{\left(\mathbb E[|\xi-\eta|^p]\right)^{\frac{1}{p}}:\,\,\text{for any $\xi\in\mathbb L^p(\mathcal{F}_T;\mu)$, $\eta\in\mathbb L^p(\mathcal{F}_T;\nu)$}\right\}.
\end{equation}
Let $C^0(\sP_2(\mathbb R^d))$ denote the set of $W_2$-continuous functions and $C^{0,1}(\sP_2(\mathbb R^d))$ denote the set of $W_2$-Lipschitz continuous functions. $U:\sP_2(\mathbb R^d)\to\R$ is said to be differentiable at $\mu\in\sP_2(\R^d)$ (cf. \cite{AmbGigSav,GanTud:19, Lions}) 
if the Wasserstein gradient $D_{\mu}U(\mu,\cdot)$ -- as an element of $\overline{\nabla C_c^\infty(\R^d)}^{L^2(\mu)}$ -- exists and one has the expansion 
\[
U(\mathcal{L}_{\xi+\eta})-U(\mathcal{L_{\xi}})=\mathbb E\left[\langle D_{\mu}U(\mu,\xi),\eta\rangle\right]+o(\|\eta\|_2),\quad \text{for any $\xi,\eta\in\mathbb L^2(\mathcal{F}_T)$.}
\]

Let $g:\M\times\sP_2(\M)\to\R$ be the terminal cost, and let $L:\M\times\R^d\times \sP_2(\M)\to\R$ be the Lagrangian function. We further make the following assumptions on $g,L$:

\begin{equation}\label{hyp:H-L}
L(\cdot,\cdot,\mu)\in C_{\rm loc}^{1,1}(\M\times\R^d), g(\cdot,\mu)\in C_{\rm loc}^{1,1}(\R^d), {\rm{uniformly\ in\ }} \mu. \tag{H\arabic{hyp}}
\end{equation}
\stepcounter{hyp}

\begin{align}\label{hyp:H-L-g-Lip_cont}
g(x,\cdot)\ {\text{is continuous in $\mu$ with\ respect\ to\ }}W_1,\ {\text{locally uniformly in }} \mathbb R^d\qquad\qquad \qquad\tag{H\arabic{hyp}}\\
\text{and } L(x,v,\cdot)\ {\text{is continuous in $\mu$ with\ respect\ to\ }}W_1,\ {\text{locally uniformly in }} \mathbb R^d\times \R^d.\nonumber
\end{align}
\stepcounter{hyp}

\begin{equation}\label{hyp:L-g_bdbl}
L(x,v,\mu)\ge \theta_1(|v|)-\theta_2(\mu)(|x|+1),\ \  g(x,\mu)\ge -\theta_2(\mu),\ \forall (x,v,\mu)\in\R^{2d}\times\sP_2(\R^d), \tag{H\arabic{hyp}}
\end{equation}
where $\theta_1:[0,+\infty)\to[0,+\infty)$ is a given superlinear function and $\theta_2:\sP_2(\R^d)\to[0,+\infty)$ is a given function which is bounded in $\{\mu\in\sP_2(\R^d): M_2(\mu)\le R \}$, for any $R>0$.
\stepcounter{hyp}

\begin{align}\label{hyp:D_xg}
&D_xg\in C^{0,1}(\mathbb R^d\times\sP_2(\mathbb R^d)). \tag{H\arabic{hyp}}
\end{align}
\stepcounter{hyp}

We suppose that $v\mapsto L(x,v,\mu)$ is strictly convex for all $(x,\mu)\in\R^d\times\sP_2(\R^d)$ and there exists $c_0>0$ such that 
\begin{equation}\label{hyp:H-conv-in-p}
D_{vv}^2L\le \frac{1}{c_0}I_d.\tag{H\arabic{hyp}}
\end{equation}
\stepcounter{hyp}

\begin{align}\label{hyp:DpH}
 D_xL,D_vL \text{ are uniformly Lipschitz continuous in }\R^d\times \R^d\times\sP_2(\R^d),\tag{H\arabic{hyp}}\\
 \left|[D_vL(0,\cdot,\mu)]^{-1} (p)\right|\le C(p),\qquad\qquad\qquad\qquad\qquad\qquad\quad\nonumber
\end{align}
where $C(p)>0$ is independent of $\mu$. 
\stepcounter{hyp}

We remark that the strict convexity assumption on $L(x,\cdot,\mu)$ implies that $[D_vL(0,\cdot,\mu)]^{-1}$ exists. The assumption $|[D_vL(0,\cdot,\mu)]^{-1}(p)|\le C(p)$ simply means that this vector field is locally bounded with respect to $p$, independently of the measure variable.
Let us emphasize that in \eqref{hyp:H-L-g-Lip_cont} and in the last part of \eqref{hyp:DpH} the continuity in the measure variable is taken with respect to the $W_1$ metric (rather than the $W_2$ one). The reason behind this is that in our consideration the natural space for the solution to the Fokker--Planck equation in \eqref{eq:mfg} will be $C([0,T]; (\sP_1(\R^d),W_1))$, and so, accordingly, the data in the Hamilton--Jacobi equation need to be continuous with respect to $W_1$. This continuity is in line with the typical assumptions in the literature (cf. \cite{CarPor}).


We impose our crucial displacement monotonicity assumptions on the terminal cost $g$ and on  the non-separable Lagrangian function $L$:

\begin{align}\label{hyp:g_mon-2}
\mathbb{E}\left\{[D_x g(X^1,\mu_1)-D_x g(X^2,\mu_2)]\cdot(X^1-X^2)\right\}\ge 0,\tag{H\arabic{hyp}}
\end{align}\stepcounter{hyp}
for any $X^1,X^2\in\mathbb L^2(\mathcal{F}_T)$ with $\mathcal{L}_{X^1}=\mu_1\in\sP_2(\mathbb R^d)$ and $\mathcal{L}_{X^2}=\mu_2\in\sP_2(\mathbb R^d)$. We recall that the notion of displacement monotonicity was proposed in \cite{GanMesMouZha}, and the previous inequality is the same as the one in \cite[Definition 2.2]{GanMesMouZha}.

\begin{align}\label{hyp:L-mon-2}
\nonumber\mathbb{E}\Big\{[D_xL(X^1,Z^1,\mu_1)&- D_xL(X^2,Z^2,\mu_2)]\cdot(X^1-X^2)\Big\}\\
&+\mathbb{E}\Big\{[D_vL(X^1,Z^1,\mu_1)-D_vL(X^2,Z^2,\mu_2)]\cdot(Z^1-Z^2)\Big\}\ge 0,\tag{H\arabic{hyp}}
\end{align}
\stepcounter{hyp}
for any $X^1,X^2,Z^1,Z^2\in\mathbb L^2(\mathcal{F}_T)$ with $\mathcal{L}_{X^1}=\mu_1\in\sP_2(\mathbb R^d)$ and $\mathcal{L}_{X^2}=\mu_2\in\sP_2(\mathbb R^d)$.

\begin{remark}
We recall that $g$ is \emph{Lasry--Lions monotone} if
\begin{align}\label{LLmon}
\mathbb{E}\left\{g(X^1,\mu_1)+ g(X^2,\mu_2)-g(X^1,\mu_2)-g(X^2,\mu_1)\right\}\ge 0,
\end{align}
for any $X^1,X^2\in\mathbb L^2(\mathcal{F}_T)$ with $\mathcal{L}_{X^1}=\mu_1\in\sP_2(\mathbb R^d)$ and $\mathcal{L}_{X^2}=\mu_2\in\sP_2(\mathbb R^d)$. Assume that $g$ is smooth enough in $x$ and $\mu$. Then the inequality \eqref{LLmon} is equivalent to 
\begin{align}\label{LLmonsmooth}
\mathbb{E}\left\{\tilde{\mathbb E}[D^2_{x\mu}g(X,\mu,\tilde X)\delta\tilde X]\cdot \delta X\right\}\ge 0,
\end{align}
for any $X,\delta X\in \mathbb L^2(\mathcal{F}_T)$ with $\mathcal{L}_X=\mu\in\sP_2(\mathbb R^d)$, where $(\tilde X,\delta\tilde X)$ is an independent copy of $(X,\delta X)$ and $\tilde{\mathbb E}$ is the (conditional) expectation corresponding to $(\tilde X,\delta\tilde X)$. Similarly, the fact that $g$ satisfies the displacement monotonicity assumption \eqref{hyp:g_mon-2} is equivalent to
\begin{align}\label{dismonsmooth}
\mathbb{E}\left\{\tilde{\mathbb E}[D^2_{x\mu}g(X,\mu,\tilde X)\delta\tilde X]\cdot \delta X+[D^2_{xx}g(X,\mu)\delta X]\cdot\delta X\right\}\ge 0,
\end{align}
for any $X,\delta X\in \mathbb L^2(\mathcal{F}_T)$ with $\mathcal{L}_X=\mu\in\sP_2(\mathbb R^d)$.

Let us consider $g:\R^d\times\sP_2(\R^d)\to\R$ defined as 
$$g(x,\mu):=C|x|^2+(\phi\star\mu)(x),$$
where $C>0$ and $\phi\in C^2(\mathbb R^d)$ with $-CI_d\leq D^2_{xx}\phi< 0$.
By \eqref{LLmonsmooth} and \eqref{dismonsmooth}, $g$ fails to be Lasry--Lions monotone while $g$ is displacement monotone. This example shows that the displacement monotonicity in general does not imply Lasry--Lions monotonicity. It is also immediate to see that if $g(x,\mu)\equiv g(x)$, then $g$ is trivially Lasry--Lions monotone. But, this function will be displacement monotone only if $g$ is convex. Thus, in general the Lasry--Lions monotonicity does not imply displacement monotonicity either. This example also shows that there are, however, functions which are both Lasry--Lions and displacement monotone in the same time.
%
%
\end{remark}

\begin{remark}\label{rmk:lift}
It is important to notice that the condition in \eqref{hyp:L-mon-2} can be naturally seen as an extension of \eqref{hyp:g_mon-2} from the state space $\R^d$ to $\R^{d}\times\R^d$. Indeed, consider the map $\tilde L:\mathbb R^d\times\mathbb R^d\times\sP_2(\mathbb R^{d}\times\mathbb R^d)\to\R$ defined by
\[
\tilde L(x,z,\nu)=L(x,z,{\pi_{1}}_\sharp\nu),\,\text{for any $(x,z,\nu)\in\mathbb R^d\times\mathbb R^d\times\sP_2(\mathbb R^{d}\times\mathbb R^d)$}
\]
where ${\pi_{1}}_\sharp\nu$ is the first marginal of $\nu$. For any $X^1,X^2,Z^1,Z^2\in\mathbb L^2(\mathcal{F}_T)$ with $\mathcal{L}_{X^1}=\mu_1$ and $\mathcal{L}_{X^2}=\mu_2$, we consider
\small
\begin{align*}
&\mathbb E\Big\{\big\langle (D_x\tilde L,D_v\tilde L)(X^1,Z^1,\mathcal{L}_{(X^1,Z^1)})-(D_x\tilde L,D_v\tilde L)(X^2,Z^2,\mathcal{L}_{(X^2,Z^2)}),(X^1-X^2,Z^1-Z^2) \big\rangle\Big\}\\
&=\nonumber\mathbb{E}\left\{[D_xL(X^1,Z^1,\mu_1)- D_xL(X^2,Z^2,\mu_2)]\cdot(X^1-X^2)\right\}\\
&+\mathbb{E}\left\{[D_vL(X^1,Z^1,\mu_1)-D_vL(X^2,Z^2,\mu_2)]\cdot(Z^1-Z^2)\right\}.
\end{align*}\normalsize

\end{remark}

\begin{lemma}\label{conv}
Suppose that \eqref{hyp:D_xg} takes place. Then \eqref{hyp:g_mon-2} implies that $g(\cdot,\mu)$ is convex on $\mathbb R^d$ for all $\mu\in\sP_2(\mathbb R^d)$.
\end{lemma}
\begin{proof}
By assumption \eqref{hyp:D_xg}, it is enough to show the convexity for any $\mu\in\sP_2(\mathbb R^d)$ that has positive density $\rho$. We further notice that the convexity of $g(\cdot,\mu)$ is equivalent to the monotonicity of $D_xg(\cdot,\mu)$.

Let us suppose the contrary, i.e. that there are two different points $x_1,x_2\in\mathbb R^d$ and $\mu_0\in\sP_2(\mathbb R^d)$ with positive density $\rho_0$ such that
\[
\langle D_{x}g(x_1,\mu_0)-D_xg(x_2,\mu_0),x_1-x_2\rangle<0.
\]
Since \eqref{hyp:D_xg} holds and $\rho_0$ is positive, there exist small $\delta_1,\delta_2>0$ such that $\mu_0(B_{\delta_1}(x_1))=\mu_0(B_{\delta_2}(x_2))$, $B_{\delta_1}(x_1)\cap B_{\delta_2}(x_2)=\emptyset$ and
for any $x\in B_{\delta_1}(x_1)$ and $y\in B_{\delta_2}(x_2)$ such that  we have
\[
\langle D_{x}g(x,\mu_0)-D_xg(y,\mu_0),x-y\rangle<0.
\]
Now, let $\xi_1\in\mathbb L^2(\mathcal{F}_T)$ such that $\mathcal{L}_{\xi_1}=\mu_0.$ Consider a transport map $T:B_{\delta_1}(x_1)\to B_{\delta_2}(x_2)$ between the measures $\mu_0\mres B_{\delta_1}(x_1)$ and $\mu_0\mres B_{\delta_2}(x_2)$ (one can simply take Brenier's map for instance). Define $A_i:=\xi_1^{-1}(B_{\delta_i}(x_i))$, $i=1,2$. We notice that $A_1,A_2\in\mathcal{F}_T$, $\mathbb{P}(A_1)=\mathbb{P}(A_2)$ and $\mathbb{P}(A_1\cap A_2)=0$. Then we consider $\xi_2\in\mathbb L^2(\mathcal{F}_T)$ defined as follows:
$$
\xi_2(\omega)=\left\{
\begin{array}{ll}
\xi_1(\omega), & \omega\notin (A_1\cup A_2),\\
(T\circ \xi_1)(\omega), & \omega\in A_1,\\
(T^{-1}\circ \xi_1)(\omega), & \omega\in A_2.
\end{array}
\right.
$$
We readily check that $\cL_{\xi_2}=\mu_0.$ Then, by construction, we find
\[
\mathbb E\Big[ \langle D_x g(\xi_1, \mathcal{L}_{\xi_1}) - D_x g(\xi_2, \mathcal{L}_{\xi_2}), \xi_1-\xi_2\rangle\Big] < 0,
\]
which contradicts with the displacement monotonicity of $g$. The result follows.
\end{proof}
\begin{remark}
The result of Lemma \ref{conv} can be seen as a slight improvement of \cite[Lemma 2.6]{GanMesMouZha} in weakening the regularity assumptions on the data.
\end{remark}

\begin{lemma}\label{conv:2}
Suppose that \eqref{hyp:DpH} takes place. Then \eqref{hyp:L-mon-2} implies that $L(\cdot,\cdot,\mu)$ is convex on $\mathbb R^d\times\mathbb R^d$ for all $\mu\in\sP_2(\mathbb R^d)$.
\end{lemma}
\begin{proof}
By the observation made in Remark \ref{rmk:lift}, we can apply Lemma \ref{conv}, to obtain that $\tilde L(\cdot,\cdot,\nu)$ is convex on $\mathbb R^d\times\mathbb R^d$ for all $\nu\in\sP_2(\mathbb R^d\times\mathbb R^d)$. This implies that $L(\cdot,\cdot,\mu)$ is convex in $\mathbb R^d\times\mathbb R^d$ for all $\mu\in\sP_2(\mathbb R^d)$.
\end{proof}

Let $H:\M\times\R^d\times \sP_2(\M)\to\R$ be the Hamiltonian function such that $L^*(x,\cdot,\mu)=H(x,\cdot,\mu)$ for all $x\in\M$ and $\mu\in\sP_2(\R^d)$ (i.e. $H$ is the Legendre-Fenchel transform of $L$ in its second variable).

\begin{remark}\label{Hassum}
Standard convex analysis theory ensures that the assumptions \eqref{hyp:H-L}, \eqref{hyp:H-L-g-Lip_cont},  \eqref{hyp:DpH}, \eqref{hyp:H-conv-in-p}, \eqref{hyp:L-mon-2} for the Lagrangian function $L$ are equivalent to the following assumptions on the corresponding Hamiltonian $H$, respectively
\begin{equation}\label{C11xp}
H(\cdot,\cdot,\mu)\in C_{\rm loc}^{1,1}(\mathbb R^d\times\mathbb R^d), {\rm{uniformly\ in\ }} \mu,
\end{equation}
\begin{equation}\label{C0muW1}
H(x,p,\cdot)\ {\text{is continuous in $\mu$ with\ respect\ to\ }}W_1, \ {\text{locally uniformly in }}\R^{d}\times \R^d 
\end{equation}

\begin{eqnarray}\label{HLip}
\left.\begin{array}{c}
D_xH,D_pH \text{ are uniformly Lipschitz continuous in }\R^d\times \R^d\times\sP_2(\R^d),\\
 |D_pH(0,p,\mu)|\le C(p)\ {\rm{and}}\ D^2_{pp}H\geq c_0I_d,
 \end{array}\right.
\end{eqnarray}
and moreover, for any $X^1,X^2,P^1,P^2\in\mathbb L^2(\mathcal{F}_T)$ with $\mathcal{L}_{X^1}=\mu_1\in\sP_2(\mathbb R^d)$ and $\mathcal{L}_{X^2}=\mu_2\in\sP_2(\mathbb R^d)$,
\begin{align}\label{Hdis}
\mathbb{E}\Big\{\big(-D_xH(X^1,P^1,\mu_1)&+D_xH(X^2,P^2,\mu_2)\big)\cdot (X^1-X^2)\Big\}\\
&\nonumber+\mathbb{E}\Big\{\left(D_pH(X^1,P^1,\mu^1)-D_pH(X^2,P^2,\mu^2)\right)\cdot\left(P^1-P^2\right)\Big\}\geq 0,
\end{align}
where $C(p)>0$ is independent of $\mu$. We note that \eqref{Hdis} implies $H$ is convex in its second variable.

The only result which might not be straightforward, is the equivalence between \eqref{hyp:L-mon-2} and \eqref{Hdis}, so let us sketch its proof. Notice that we have the Legendre-Fenchel inequality: for all $x\in\M$, $\mu\in\sP_2(\R^d)$ and $p,v\in\R^d$ we have $H(x,p,\mu)+L(x,v,\mu)\ge p\cdot v.$ It is well-known that we have the equality if and only if $v=D_pH(x,p,\mu)$ or $p=D_v L(x,v,\mu)$. As a consequence $D_vL(x,\cdot,\mu)=[D_pH(x,\cdot,\mu)]^{-1}$. Furthermore
$$D_xH(x,p,\mu)=-D_x L(x,D_p H(x,p,\mu),\mu).$$
Supposing that \eqref{hyp:L-mon-2} takes place, fix $X^1,X^2,P^1,P^2\in\mathbb L^2(\mathcal{F}_T)$ with $\mathcal{L}_{X^1}=\mu_1\in\sP_2(\mathbb R^d)$ and $\mathcal{L}_{X^2}=\mu_2\in\sP_2(\mathbb R^d)$. Then, by setting $Z^i:=D_pH(X^i,P^i,\mu_i)$, $i=1,2$, and noticing that by the Lipschitz continuity assumption on $D_pH$, $Z^i\in \mathbb L^2(\mathcal{F}_T)$, we obtain \eqref{Hdis}, by using \eqref{hyp:L-mon-2} for $X^1,X^2,Z^1,Z^2$. The converse implication can be checked similarly.

\end{remark}

In what follows we show that the displacement monotonicity assumption \eqref{Hdis} imposed on $H$ (and hence the condition \eqref{hyp:L-mon-2} imposed on $L$) is implied by the corresponding displacement monotonicity assumption, proposed in \cite[Definition 3.4]{GanMesMouZha}. Therefore, our standing assumptions in this manuscript are in general weaker than the ones from \cite{GanMesMouZha}.

\begin{lemma}\label{lem:example}
Assume that $H\in \cC^2(\mathbb R^d\times\mathbb R^d\times\sP_2(\mathbb R^d))$, $D_{pp}H\geq c_0 I_d$ for some $c_0>0$, and assume further that $D^2_{x\mu}H,D^2_{xx}H, D^2_{pp}H, D^2_{p\mu}H$ are uniformly bounded. Then \eqref{Hdis} holds if, for any $X,\delta X,P\in\mathbb L^2(\mathcal{F}_T)$ with $\mathcal{L}_{X}=\mu\in\sP_2(\mathbb R^d)$,
\begin{align}\label{Hdisdiff}
\mathbb{E}\Big[\langle\tilde {\mathbb E}\big[D^2_{x\mu}H(X,P,\mu,\tilde X)\delta \tilde X\big],\delta X\rangle&+\langle D^2_{xx}H(X,P,\mu)\delta X,\delta X\rangle\Big]\\
&\nonumber\leq-\frac{1}{4}\mathbb E\left[\left|[D^2_{pp}H(X,P,\mu)]^{-\frac{1}{2}}\tilde{\mathbb E}\left[D^2_{p\mu}H(X,P,\mu,\tilde X)\delta \tilde X\right]\right|^2\right].\nonumber
\end{align}
where $\tilde {\mathbb E}$ is a (conditional) expectation corresponding to $\tilde X$ and $\delta \tilde X$.
\end{lemma}
\begin{proof}
Let $\bar X:=X_1-X_2$, $\bar{ P}:=P_1-P_2$, $X_{\theta}:=X_2+\theta \bar X$ and $P_{\theta}:=P_2+\theta \bar P$. Then
\begin{eqnarray*}
&&\mathbb E\Big[\langle D_xH(X_1,P_1,\mu_1)-D_xH(X_2,P_2,\mu_2), \bar X\rangle\\
&&-\langle D_pH(X_1,P_1,\mu_1)-D_pH(X_2,P_2,\mu_2),\bar P\rangle\Big]\\
&=&\mathbb E\Big[\int_0^1\langle D^2_{xx}H(X_\theta, P_{\theta}, \mathcal{L}_{X_\theta})\bar X+\tilde{\mathbb E}[D^2_{x\mu}H(X_\theta,P_\theta,\mathcal{L}_{X_{\theta}},\tilde X_\theta)\tilde{\bar X}]+D^2_{xp}H(X_\theta,P_\theta,\mathcal{L}_{X_\theta})\bar P,\bar X\rangle\\
&&-\langle D^2_{px}H(X_\theta,P_\theta,\mathcal{L}_{X_\theta})\bar X+\tilde {\mathbb E}[D^2_{p\mu}H(X_\theta,P_{\theta},\mathcal{L}_{X_\theta},\tilde X_\theta)\tilde{\bar X}]+D^2_{pp}H(X_\theta,P_\theta,\mathcal{L}_{X_\theta})\bar P,\bar P\rangle d\theta\Big]\\
&=&\mathbb E\Big[\int_0^1\langle D^2_{xx}H(X_\theta,P_{\theta},\mathcal{L}_{X_\theta})\bar X+\tilde {\mathbb E}[D^2_{x\mu}H(X_\theta,P_\theta,\mathcal{L}_{X_\theta},\tilde X_\theta)\tilde{\bar X}],\bar X\rangle\\
&&+\frac{1}{4}\left|[D^2_{pp}H(X_\theta,P_\theta,\mathcal{L}_{X_\theta})]^{-\frac{1}{2}}\tilde {\mathbb E}[D^2_{p\mu}H(X_\theta,P_\theta,\mathcal{L}_{X_\theta},\tilde X_\theta)\tilde{\bar X}]\right|^2\\
&&-\left|D^2_{pp}H(X_\theta,P_\theta,\mathcal{L}_{X_\theta})^{\frac{1}{2}}\bar P+\frac{1}{2}D^2_{pp}H(X_\theta,P_\theta,\mathcal{L}_{X_\theta})^{-\frac{1}{2}}\tilde{\mathbb E}[D^2_{p\mu}H(X_\theta,P_\theta,\mathcal{L}_{\theta},\tilde X_\theta)\tilde{\bar X}]\right|^2d\theta\Big]\leq 0.
\end{eqnarray*}
\end{proof}
\begin{remark}
\begin{enumerate}
\item A quite general class of Hamiltonians constructed in \cite[Lemma 3.8]{GanMesMouZha} satisfies all the assumptions in Lemma \ref{lem:example} including \eqref{Hdisdiff}. Then \eqref{Hdis} holds by Lemma \ref{lem:example}. Moreover, it can be easily verified that the class also satisfies \eqref{C11xp}-\eqref{Hdis}. Therefore, the corresponding class of Lagrangians satisfies \eqref{hyp:H-L}-\eqref{hyp:L-g_bdbl}, \eqref{hyp:DpH}, \eqref{hyp:H-conv-in-p}, \eqref{hyp:L-mon-2}.\\ 
\item More particularly, the following model Hamiltonians satisfy our assumptions. Let $H_0:\R^d\times\R^d\times\sP_2(\R^d)\to\R$ be of class $C^2$ in the first two variables such that there exists $C_0>0$ with the property 
$$
|\partial_x^\a\partial_p^\beta H_0(x,p,\mu)|< C_0,\ {\rm{for\ all\ }}(x,p,\mu)\in\R^d\times\R^d\times\sP_2(\R^d),
$$ 
and for all $\alpha,\beta\in (\N\cup \{0\})^d$ multi-indices with $0\le |\a|+|\b|\le 2$. Furthermore, assume that $H_0(x,p,\cdot)$ is continuous with respect to $W_1$ and $D_p H_0$ is Lipschitz continuous in the last variable with respect to $W_1$, uniformly with respect to $(x,p)$. Then we define $H:\R^d\times\R^d\times\sP_2(\R^d)\to\R$ as 
$$
H(x,p,\mu):= H_0(x,p,\mu) +\frac{C_0}{2}(|p|^2-|x|^2).
$$

\item It is not hard to see that if the Hamiltonian is sufficiently regular and separable (i.e. $D^2_{p\mu} H = 0$), the condition \eqref{Hdisdiff} is equivalent to \eqref{Hdis}. Indeed, we can see this from the last line of the proof of Lemma \ref{lem:example}, by choosing $\bar P=0$. \\
\item If the Hamiltonian is non-separable (i.e. $D^2_{p\mu} H \neq 0$), then the monotonicity condition \eqref{Hdis} in general does not imply \eqref{Hdisdiff}, and so the former one is weaker than the latter one. Indeed, looking again at the last line of the proof of Lemma \ref{lem:example}, we see that this in general does not vanish and it gives a negative contribution.
\end{enumerate}
\end{remark}

\section{Existence of a solution when $\beta=0$}\label{sec:existence}

In this section we provide the result on the existence of a classical solution to the MFG system \eqref{eq:mfg} when $\beta=0$. We shall emphasize that the convexity properties of $x\mapsto g(x,\mu)$ and $(x,v)\mapsto L(x,v,\mu)$ for all $\mu\in\sP_2(\R^d)$ (implied by the displacement monotonicity assumptions, cf. Lemmas \ref{conv} and \ref{conv:2}) play an important role in showing such a result and this seems to be new in the literature. 

The study of fully convex control problems received a great attention in the past, and this goes back to the works by Rockafellar in the 1970s (cf. \cite{Roc:70-1,Roc:70-2,Roc:71}). In these works, for deterministic optimal control problems, powerful duality techniques were developed and the author could handle even non-smooth (but convex) data. Later, in this fully convex setting, many other results followed (cf. \cite{RocWol:00-1,RocWol:00-2,GoeRoc,Goe:05-2}). In particular, in \cite{Goe:05-1} (see also \cite{BarEva} for special Hamiltonians) it was proven that in the case of fully convex control problem involving data of class $C^{1,1}$, the associated value function (which is convex in the position variable) is also of class $C^{1,1}$ in the position variable.

This regularity on the value function will also hold in our context, which will in turn imply that the drift for the continuity equation will also be Lipschitz continuous in the position variable. By this (using the regularity property of $H$ in the measure variable), we can build a suitable fixed point scheme that would yield the existence of a solution to \eqref{eq:mfg} when $\beta=0$. So, in fact the monotonicity conditions \eqref{hyp:g_mon-2}  and \eqref{hyp:L-mon-2} are not used explicitly in this section.
 
Furthermore, for the results of this section, one can slightly weaken the first part of assumptions from \eqref{hyp:DpH} (or equivalently the first part of the one in \eqref{HLip}). In particular, for the existence of a solution to \eqref{eq:mfg} when $\beta=0$, we do not need to impose Lipschitz continuity assumptions on $D_xL$ (or on $D_xH$). Inspired by the assumptions from \cite[Theorem 3.3]{GanSwi:14-particles}, we impose the following condition on $D_pH$. 

Let $b:\R^d\to\R^d$ be a Lipschitz continuous vector field such that there exists $C_b>0$ with $|b(x)|\le C_b(1+|x|)$ for all $x\in\R^d$. Define $V_b:\R^d\times\sP_2(\R^d)\to\R^d$ by $V_b:=D_pH(x,b(x),\mu)$. 
Suppose that for all $R>0$, there exists $\omega_R:[0,+\infty)\to[0,+\infty)$, a modulus of continuity, with  $\omega_R(s)>\omega_R(0)=0$ for all $s>0$ and $\int_0^{s_0}\frac{\dd s}{\omega_R(s)}=+\infty$ for some $s_0>0$ such that for any $b$ with the above mentioned properties we have
\begin{equation}\label{D_pHLip-alternative}
\ds\iint_{\R^d\times\R^d} [V_b(x,\mu^1)-V_b(x,\mu^2)]\cdot (x-y)\dd\gamma(x,y)\le \omega_R(W_2(\mu^1,\mu^2))W_2(\mu^1,\mu^2),
 \tag{H6'}
\end{equation}
for any $\mu^1,\mu^2\in\left\{\mu\in\sP_2(\R^d):\ M_2(\mu)\le R\right\}$ and any $\gamma\in\Pi_o(\mu^1,\mu^2)$ (here $\Pi_o(\mu^1,\mu^2)$ stands for the set of optimal plans realizing $W_2(\mu^1,\mu^2)$). Suppose also that there exists $C_H>0$ and $\omega:[0,+\infty)\to[0,+\infty)$ continuous increasing with $\omega(0)=0$ such that
\begin{eqnarray}\label{D_pH-W_1}
\begin{array}{r}
|D_pH(x,p,\mu_1)-D_pH(x,p,\mu_2)|\le C_H(|x|+|p|+1)\omega(W_1(\mu_1,\mu_2)), \ {\rm{if}}\ W_1(\mu_1,\mu_2)\ll 1,\\
{\rm{and}}\\
|D_pH(x_1,p_1,\mu)-D_pH(x_2,p_2,\mu)|\le C_H(|x_1-x_2|+|p_1-p_2|),\ \ \forall (x_1,x_2,p_1,p_2)\in\R^d,\ \mu\in\sP_2(\R^d).
\end{array}
\end{eqnarray}
We suppose furthermore that the second part of \eqref{HLip} takes place, i.e.
\begin{equation}\label{D_pH-bound}
 |D_pH(0,p,\mu)|\le C(p)\ {\rm{and}}\ D^2_{pp}H\geq c_0I_d.\tag{H6''}
\end{equation}

\begin{remark}
We would like to emphasize that \eqref{D_pHLip-alternative} and \eqref{D_pH-W_1} are a relaxation of \eqref{HLip}. Indeed, \eqref{D_pHLip-alternative} and \eqref{D_pH-W_1} would allow us to consider $H(x,p,\mu):=a(\mu)\frac{|p|^2}{2}$, where we suppose that there exists $C_0>0$ such that $\frac{1}{C_0}\le a(\mu)\le C_0$, for all $\mu\in\sP_s(\R^d)$ and $a:\sP_2(\R^d)\to\R$ is Lipschitz continuous with respect to $W_1$. This Hamiltonian satisfies also the other assumptions imposed in this section, but it clearly does not satisfy the first part of \eqref{HLip}. However, this Hamiltonian is not displacement monotone, so the results of the next section do not apply for this example.
\end{remark}

There are some existence results to \eqref{eq:mfg} already in the literature when $\beta=0$, however these are for weak solutions, without any monotonicity assumption, see \cite{Cardaliaguet,CarPor,CarSou}.  When $\beta^2>0$, the non-degeneracy gives enough compactness for the existence of classical solutions to the MFG system \eqref{eq:mfg} without any monotonicity assumptions, see e.g. for analytical arguments (cf. \cite{Cardaliaguet, CarDelLasLio, CarPor}) and for probabilistic arguments (cf. \cite{ChaCriDel}). Therefore, we shall only focus on the case $\beta=0$ regarding the existence result.

Before we show the existence result, let us define in which sense do we understand a pair $(u,\rho)$ to be a solution to the mean field game system \eqref{eq:mfg}.
\begin{definition}
We say that $(u,\rho)$ is a solution pair to the mean field game system \eqref{eq:mfg} if \\
(i) $u$ is locally Lipschitz continuous, solves the Hamilton-Jacobi-Bellman equation in \eqref{eq:mfg} in the viscosity sense and $D^2_x u(t,\cdot)$ is essentially bounded on $\R^d$, uniformly with respect to $t\in[0,T]$;\\
(ii) $\rho$ solves the Fokker-Planck equation in \eqref{eq:mfg} in the distributional sense and $(\rho_t)_{t\in[0,T]}$ is a continuous curve in the metric space $(\sP_1(\mathbb R^d),W_1)$.
\end{definition}

\begin{remark}
It is important to underline that the uniform bound on $D^2_x u(t,\cdot)$ in the notion of solution is a consequence of the convexity of $x\mapsto g(x,\mu)$ and $(x,v)\mapsto L(x,v,\mu)$ for all $\mu\in\sP_2(\R^d)$ and this bound is independent of the intensity of the noise $\b^2\ge 0$. In the case when $\b=0$, this estimate will further imply (see Lemma \ref{lem:HJ} below) that $u$ is a classical solution to the Hamilton--Jacobi--Bellman equation.
\end{remark}

\begin{lemma}\label{lem:HJ}
Let $\rho:[0,T]\to \cP_2(\R^d)$ be a given continuous curve with respect to $W_1$ and let $C_\rho>0$ such that $M_2(\rho_t)\le C_\rho$ for all $t\in[0,T]$. Let us suppose that the assumptions \eqref{hyp:H-L}-\eqref{hyp:L-g_bdbl}
 take place. Suppose furthermore that $x\mapsto g(x,\mu)$ and $(x,v)\mapsto L(x,v,\mu)$ are convex for all $\mu\in\sP_2(\R^d).$
 
Then the problem
\begin{equation}\label{eq:HJ}
\left\{
\begin{array}{ll}
-\partial_t u(t,x) + H(x,-D_xu(t,x),\rho_t)=0,&  {\rm{in}}\ (0,T)\times\M,\\[3pt]
u(T,x)=g(x,\rho_T), & {\rm{in}}\ \M,
\end{array}
\right.
\end{equation}
has a unique viscosity solution $u:[0,T]\times\R^d\to\R$, which is continuously differentiable. Furthermore, this solution satisfies the following derivative estimates: there exists $C_0>0$ (depending only on the data and $T$), and for all $R>0$ there exist $C_1=C_1(R)>0$ (depending on $R$, the data and $T>0$) and $C_2=C_2(R,\rho)>0$ (depending on $R>0$, $C_\rho$, the data and $T>0$), such that 
\begin{itemize}
\item[(i)] $|D_x u(t,x)|\le C_1$  for all $(t,x)\in [0,T]\times B_R$;
\item[(ii)]  $|\partial_t u(t,x)|\le C_2$ for all $(t,x)\in [0,T]\times B_R$;
\item[(iii)] $|D^2_xu(t,x)|\le C_0$ for all $t\in [0,T]$ and a.e. $x\in\R^d$;
\item[(iv)] $|\partial_t D_x u(t,x)|\le C_2$, 
for a.e. $(t,x)\in [0,T]\times B_R$.
\end{itemize}
\end{lemma}

\begin{proof}
It is well-known that solutions to \eqref{eq:HJ} are intimately linked to value functions in optimal control problems. Under our standing assumptions, classical results (cf. \cite{CanSin}) imply that the unique viscosity solution of \eqref{eq:HJ} can be obtained as the value function of the optimal control problem. Given $\rho$ as stated in the Lemma \ref{lem:HJ},
\begin{align*}
u(t,x)=\inf\left\{\int_t^T L(X_s,\alpha_s,\rho_s)\dd s+ g(X_T,\rho_T)\right\},\ \ {\rm{s.\ t.}}\ \ 
\left\{
\begin{array}{ll}
\dd X_s=\alpha_s\dd s, & s\in(t,T),\\
X_t = x,
\end{array}
\right.
\end{align*}
where the infimum is taken over all $\alpha\in L^2([t,T])$.

It is standard (cf. \cite{CanSin}) to show that $u$ is continuous in $t$, $u(t,\cdot)$ is locally Lipschitz continuous on $\R^d$ and $u(t,\cdot)$ is semi-concave, uniformly with respect to $t\in[0,T]$.

These arguments yield (i).

\medskip

{\it Claim.} $u(t,\cdot)$ is convex, uniformly in time.

\medskip

{\it Proof of Claim.} Let $x^1,x^2\in\R^d$ and let $\l\in[0,1]$. For $\e>0$, let $\a^{i,\e}$, $i=1,2$ be $\e$-optimal controls and let $X^{i,\e}$ be the corresponding paths, so we have
$$
\int_t^T L(X^{i,\e}_s,\alpha^{i,\e}_s,\rho_s)\dd s+ g(X^{i,\e}_T,\rho_T)\le u(t,x^i)+\e.
$$
Let us set $Y^\l=(1-\l)X^{1,\e}+\l X^{2,\e}$, so in particular $Y^\l_0=(1-\l)x^1+\l x^2$. We notice that $Y^\l$ is an admissible competitor for $u(t,(1-\l)x^1+\l x^2)$. 
We have 
\begin{align*}
u(t,(1-\l) x^1+\l x^2)&\le \int_t^T L(Y^{\l}_s,(1-\l)\alpha^{1,\e}_s+\l\alpha^{2,\e}_s,\rho_s)\dd s+ g(Y^{\l}_T,\rho_T)\\
&\le (1-\l)u(t,x^1) + \l u(t,x^2) - 2\e
\end{align*}
where, in the last inequality we have used the convexity of $L(\cdot,\cdot,\mu)$ and $g(\cdot,\mu)$ and the $\e$-optimality of the curves $X^{i,\e}$. By the arbitrariness of $\e>0$, we conclude about the the convexity of $u(t,\cdot)$ and the claim follows. Together with the semi-concavity result this implies (iii).

\medskip

Now, from the Hamilton--Jacobi equation, we find that $\partial_t u$ must be locally bounded. Thus, (ii) follows.

\medskip

We further differentiate the equation with respect to $x$ and find that $\partial_tD_x u$ must be locally bounded, which implies in particular that $D_x u$ is Lipschitz continuous with respect to $t$ (locally uniformly with respect to $x$). Looking again at the equation, this means that $\partial_t u$ must be continuous. The corresponding constants in the estimates are such as they are specified in the statement of the theorem.
\end{proof}

\begin{remark}
It is immediate to see that the implication of the Lemma \ref{lem:HJ} and the estimates (i)-(iii) (except point (iv)) remain valid also in the case of $\b^2>0$ (using again some classical results, cf. \cite{BuckCanQui}). The constants in those estimates are independent of $\b$.
\end{remark}

The following result will not surprise experts in optimal transport theory. However, for completeness we supply its proof here. 

\begin{lemma}\label{lem:continuity}
Let $T>0$ and $V:[0,T]\times\R^d\times\sP_2(\R^d)\to\R^d$ be a given continuous vector field and suppose that there exists $C_V>0$ such that $x\mapsto V(t,x,\mu)$ is Lipschitz continuous with constant $C_V$, uniformly in $(t,\mu)$,
\begin{align}\label{hyp:V}
V(t,\cdot,\cdot) {\rm{\ satisfies\ }} \eqref{D_pHLip-alternative},\ {\rm{with\ some\ }}\omega_R,\ \forall R>0,\ \forall t\in [0,T], 
\end{align}
and
$$
|V(t,0,\mu)|\le C_V,\ \forall t\in[0,T],\ \forall\mu\in\sP_1(\R^d)
$$
and there exists $\omega:[0,+\infty)\to[0,+\infty)$ continuous increasing with $\omega(0)=0$ such that
\begin{equation}\label{hyp:V_new}
|V(t,x,\mu)-V(t,x,\nu)|\le C_V(|x|+1)\omega(W_1(\mu,\nu)), \ {\rm{if}}\ W_1(\mu,\nu)\ll 1.
\end{equation}

Then, for any $\rho_0\in\sP_2(\R^d)$ the problem 
\begin{align}\label{eq:continuity}
\left\{
\begin{array}{ll}
\ds\partial_t\rho_t+D_x\cdot(\rho_t V(t,x,\rho_t))=0, & {\rm{in}}\ \sD'((0,T)\times\M),\\[3pt]
\rho(0,\cdot)=\rho_0.
\end{array}
\right.
\end{align}
has a unique solution. Moreover,  we have that there exists $C=C(T,C_V,M_2(\rho_0))>0$ such that 
\begin{align*}
M_2(\rho_t)\le C,\ \forall t\in[0,T] \ \ {\rm{and}}\ \ W_1(\rho_t,\rho_s)\le C|s-t|,\ \forall s,t\in[0,T].
\end{align*}
\end{lemma}

\begin{proof}
First, let us notice that by the assumption on $V$, we have
\begin{align*}
|V(t,x,\mu)|\le |V(t,0,\mu)|+ C_V|x|\le C_V(1+|x|).
\end{align*}

{\it Existence.} We define the operator $\cS:C([0,T];(\sP_1(\R^d),W_1))\to C([0,T];(\sP_1(\R^d),W_1))$ as follows. For $\tilde\rho\in C([0,T];(\sP_1(\R^d),W_1))$, we set $\cS(\tilde\rho):=\rho$, where $\rho$ is the unique solution to the problem 
\begin{equation}\label{eq:cont_existence}
\left\{
\begin{array}{ll}
\partial_t\rho_t+D_x\cdot(\rho_t V(t,x,\tilde\rho_t))=0, & {\rm{in}}\ \sD'((0,T)\times\M),\\[3pt]
\rho(0,\cdot)=\rho_0.
\end{array}
\right.
\end{equation}
The well-posedness of this is the consequence of classical results, since by the assumptions, $V(\cdot,\cdot,\tilde\rho_t)$ is Lipschitz continuous in space and continuous in time.  Now, let us show that the range of $\cS$ is a compact subset of $C([0,T];(\sP_1(\R^d),W_1))$.

\medskip

Claim 1. $M_2(\rho_t)$ is uniformly bounded if $t\in[0,T]$ (independently of $\tilde\rho$).

\medskip

Proof of Claim 1. Because of the regularity on $V(\cdot,\cdot,\tilde\rho_t)$, the solution of \eqref{eq:cont_existence} can be represented along the flow of the vector field, i.e.

\begin{align*}
\left\{
\begin{array}{ll}
\dot X(t,x)=V(t,X(t,x),\tilde\rho_t), & t\in(0,T),\\[3pt]
X(0,x)=x, & x\in\R^d.
\end{array}
\right.
\end{align*}
First, let us notice that 
\begin{align*}
|X(t,x)|\le |x|+\int_0^t C_V(1+|X(s,x)|)\dd s=|x|+C_Vt+C_V\int_0^t|X(s,x)|\dd s,
\end{align*}
thus Gr\"onwall's inequality yields
\begin{equation}\label{eq:X_growth}
|X(t,x)|\le (|x|+C_Vt)e^{tC_V},
\end{equation}
which further implies 
\begin{align}\label{eq:mom1}
|X(t,x)|^2\le 2(|x|^2+C_V^2t^2)e^{2tC_V}.
\end{align}
Since $\rho_t=X(t,\cdot)_\sharp\rho_0$, for any $\vphi\in C_b(\R^d)$ we have
\begin{align}\label{eq:mom2}
\int_{\R^d}\vphi(x)\dd\rho_t(x)=\int_{\R^d}\vphi(X(t,x))\dd\rho_0(x).
\end{align}

For $R>0$ we consider $\vphi_R\in C_b(\R^d)$ defined as $\vphi_R(x):=\min\{R^2,|x|^2\}$. Clearly, $\vphi_R\to |x|^2$, locally uniformly as $R\to+\infty$. \eqref{eq:mom2} and \eqref{eq:mom1} yield that there exists a constant $\tilde C>0$ (depending only on $C_V$ and $T$) such that 
\begin{align*}
\int_{\R^d}\vphi_R(x)\dd\rho_t(x)&=\int_{\R^d}\vphi_R(X(t,x))\dd\rho_0(x)\le \int_{\R^d}|X(t,x)|^2\dd\rho_0\\
&\le \tilde C +\tilde C\int_{\R^d}|x|^2\dd\rho_0=\tilde C(1+M_2^2(\rho_0)). 
\end{align*}
Now, by the dominated convergence theorem, as $R\to +\infty$, we have that
$$
M_2(\rho_t)\le \tilde C(1+M_2(\rho_0)),
$$
as desired.


\medskip

Claim 2.  There exists $C>0$ (independent of $\tilde\rho$) such that
$$
W_1(\rho_t,\rho_s)\le C |t-s|,\ \forall t,s\in[0,T].
$$

\medskip

Proof of Claim 2. Let us suppose that $0\le s\le t\le T$. 

\medskip

Let $\vphi$ be $1-{\rm{Lip}}(\R^d).$  Then we have
\begin{align*}
\int_{\R^d}\vphi(x)\dd(\rho_t-\rho_s)(x)&=\int_s^t\int_{\R^d}D_x\vphi\cdot V(\tau,x,\tilde\rho_\tau)\dd\rho_\tau(x)\dd \tau\\
&\le\int_s^t\int_{\R^d}C_V(1+|x|)\dd\rho_\tau(x)\dd\tau\\
&\le \int_s^t [C_V+M_2(\rho_\tau)]\dd\tau\le C|t-s|.
\end{align*}
Let us remark that all the integrals are finite by the second moment bounds on $\rho$. Now, taking supremum with respect to $\vphi$, $1-{\rm{Lip}}(\R^d)$ one obtains
$W_1(\rho_t,\rho_s)\le C|t-s|,$ so the claim follows.

\medskip

From Claim 1 and 2 we can conclude that the range of $\cS$ is compact. The continuity of $\cS$ is straightforward. Indeed, let $(\tilde\rho^n)_{n\in\N}$ be a sequence uniformly converging to $\tilde\rho$ in $C([0,T];(\sP_1(\R^d),W_1))$ as $n\to+\infty$. Let $\vphi$ be $1-{\rm{Lip}}(\R^d)$, set $\rho^n=\cS(\tilde\rho^n)$ and $\rho=\cS(\tilde\rho)$ and let $X^n$ and $X$ stand for the flows of the vector fields $V(\cdot,\cdot,\tilde\rho^n)$ and $V(\cdot,\cdot,\tilde\rho)$, respectively.

Then, one obtains
\begin{align}\label{ineq:cont_S}
\int_{\R^d}\vphi(x)\dd(\rho_t^n&-\rho_t)(x)=\int_{\R^d}[\vphi(X^n(t,x))-\vphi(X(t,x))]\dd\rho_0(x)\dd \tau\\
\nonumber&=\int_{\R^d}\int_0^1D_x\vphi(sX^n(t,x)+(1-s)X(t,x))\cdot[X^n(t,x)-X(t,x)]\dd s \dd\rho_0(x)\\
\nonumber&\le\int_{\R^d}|X^n(t,x)-X(t,x)| \dd\rho_0(x)
\end{align}
Now, by the ODEs satisfied by $X^n$ and $X$ we get
\begin{align*}
|X^n(t,x)-X(t,x)|&\le \int_0^t |V(\tau,X^n(\t,x),\tilde\rho^n)-V(\tau,X(\t,x),\tilde\rho)|\dd\t\\
&\le \int_0^t |V(\tau,X^n(\t,x),\tilde\rho^n_\t)-V(\t,X(\t,x),\tilde\rho^n_\t)|\dd\t\\
& + \int_0^t |V(\t,X(\t,x),\tilde\rho^n_\t)-V(\tau,X(\t,x),\tilde\rho)|\dd\t\\
&\le  C_V\int_0^t |X^n(\t,x)-X(\t,x)|\dd\t\\
&+\int_0^t C_V(|X(\t,x)|+1)\omega(W_1(\tilde\rho^n_\t,\tilde\rho_\t))\dd\t
\end{align*}
where in the last inequality, we used the assumptions on the field $V$. Using \eqref{eq:X_growth}, the previous chain of inequalities can be further estimated as
\begin{align*}
|X^n(t,x)-X(t,x)|&\le  C_V\int_0^t |X^n(\t,x)-X(\t,x)|\dd\t\\
&+\int_0^t C_V[(|x|+C_VT)e^{TC_V}+1]\omega\Big(\sup_{s\in[0,T]}W_1(\tilde\rho^n_s,\tilde\rho_s)\Big)\dd\t.
\end{align*}
By denoting $a_n:=\omega\Big(\sup_{s\in[0,T]}W_1(\tilde\rho^n_s,\tilde\rho_s)\Big)$, there exists a constant $C=C(T,C_V)$ such that 
\begin{align*}
|X^n(t,x)-X(t,x)|&\le  C_V\int_0^t |X^n(\t,x)-X(\t,x)|\dd\t\\
&+C(|x|+1)a_n.
\end{align*}
We notice that by assumption $\sup_{s\in[0,T]}W_1(\tilde\rho^n_s,\tilde\rho_s)\to 0$ as $n\to+\infty$, so $a_n\to 0$ as $n\to+\infty$, as well.

Gr\"onwall's inequality yields that
$$
|X^n(t,x)-X(t,x)|\le C(|x|+1)a_n e^{tC_V}\le C(|x|+1)a_n e^{TC_V}.
$$
Thus, 
$$
\lim_{n\to+\infty}\sup_{t\in[0,T]}\int_{\R^d}|X^n(t,x)-X(t,x)|\dd\rho_0(x)=0.
$$
So, by taking supremum with respect to $\vphi$ in \eqref{ineq:cont_S}, we can conclude that
$$
\lim_{n\to+\infty}\sup_{t\in[0,T]}W_1(\rho^n_t,\rho_t)=0,
$$
and so, the continuity of $\cS$ follows. So, finally, one can use Schauder's fixed point theorem to conclude that $\cS$ has a fixed point and therefore \eqref{eq:continuity} has a solution.


%

\medskip

{\it Uniqueness.}

By \eqref{hyp:V}, the vector field $V$ satisfies the assumptions in \cite[Theorem 3.3]{GanSwi:14-particles}, and therefore the uniqueness of solutions to \eqref{eq:continuity} follows from there.

\end{proof}

Now, we are in position to state the main result of this section.

\begin{theorem}\label{thm:exist_deter}
We suppose that all the assumptions \eqref{hyp:H-L}-\eqref{hyp:H-conv-in-p},  \eqref{D_pHLip-alternative}, \eqref{D_pH-bound} and \eqref{D_pH-W_1} take place and the functions $x\mapsto g(x,\mu)$ and $(x,v)\mapsto L(x,v,\mu)$ are convex. Then the mean field game system \eqref{eq:mfg}, with $\beta=0$, has a solution pair $(u,\rho)$.
\end{theorem}

\begin{proof}
Let $\rho\in C([0,T];(\sP_1(\R^d),W_1))$ be given with $\rho|_{t=0}=\rho_0$. 
Let $u$ be the unique classical solution to \eqref{eq:HJ} provided in Lemma \ref{lem:HJ}. Now set 
$$
V(t,x,\mu):=D_pH(x,-D_x u(t,x),\mu).
$$
Clearly, by our standing assumptions and the results in Lemma \ref{lem:HJ}, $V$ satisfies \eqref{hyp:V} and \eqref{hyp:V_new} with a constant $C_V>0$ and $|V(t,0,\mu)|\le C_V$, where $C_V$ depends only on the data and $|D_x u(t,0)|$ (from \eqref{HLip}), but clearly $|D_x u(t,0)|$ depends only on the previous constants and therefore $C_V$ depends only on the data. We also also have $|D_x u(t,x)|\le C(|x|+1)$, for some constant $C>0$ depending on $T$ and the data.

Let $\ov\rho\in C([0,T];(\sP_1(\R^d),W_1))$ be the unique solution of \eqref{eq:continuity} starting at $\rho_0$ with the previously set vector field $V$. So, if one considers the mapping $S:C([0,T];(\sP_1(\R^d),W_1))\to C([0,T];(\sP_1(\R^d),W_1))$ such that $S(\rho)=\ov\rho$, this is well-defined.

We show now that $S$ is a continuous mapping. Let us take a sequence $(\rho^n)_{n\in\N}$ from the space $C([0,T];~(\sP_1(\R^d),W_1))$ that uniformly converges to some $\rho\in C([0,T];(\sP_1(\R^d),W_1))$ as $n\to+\infty$. 
If we consider the corresponding unique solutions $(u^n)_{n\in\N}$ to \eqref{eq:HJ}, we find that all the bounds on $\partial_t u^n$, $D_x u^n$ and $D^2_{xx}u^n$, as stated in Lemma \ref{lem:HJ} are independent of $n$, and they just depend on the data. So, by the continuity assumptions on $H$ (transferred from the regularity assumptions on $L$) and $g$ in the measure variable, standard results on stability of viscosity solutions to Hamilton-Jacobi equations 
yield that $(u^n)_{n\in\N}$ converges locally uniformly to the unique solution of \eqref{eq:HJ} (where as data, we consider the limit curve $\rho$).

Moreover, up to passing to a subsequence that we do not relabel, $(D_x u^n)_{n\in\N}$ converges locally uniformly to $D_x u$ on $\R^d$, uniformly with respect to $t$. So, the corresponding vector fields $V^n(t,x,\mu)=D_pH(x,-D_x u^n(t,x),\mu)$ also converge locally uniformly to $V(t,x,\mu)=D_pH(x,-D_x u(t,x),\mu)$ as $n\to+\infty$. Therefore, since the sequence of curves $(\rho^n)_{n\in\N}$ is uniformly Lipschitz continuous (with respect to $W_1$) and as their second moments are uniformly bounded (as provided in Lemma \ref{lem:continuity}), Arzel\`a-Ascoli's theorem yields the existence of a subsequence that converges uniformly to some $\tilde\rho$. However, passing to the limit the continuity equation, one must have that this limit $\tilde\rho$ is the solution of the equation, when we consider $u$. By uniqueness of solutions, one must have that $\tilde\rho=\rho$. So, the continuity of $S$ follows.

Now, it remains to show that $S$ satisfies the assumptions of Schauder's fixed point theorem. Clearly, the space of curves in $C([0,T];(\sP_1(\R^d),W_1)$ that start at the fixed $\rho_0$ is convex. Moreover, because of the results provided in Lemma \ref{lem:continuity}, we find that the image of $C([0,T];(\sP_1(\R^d),W_1)$ through $S$ is the space of curves that are uniformly Lipschitz continuous with respect to $W_1$ and such that their second moments are uniformly bounded by the a constant that depends only on the data and $M_2(\rho_0)$. Therefore, this image space is compact. So, Schauder's fixed point theorem yields the existence of a fixed point of $S$, and therefore the existence of a solution to \eqref{eq:mfg} follows. 

The fact that $D^2_x u(t,\cdot)\in L^\infty(\R^d\times\R^d)$, uniformly with respect to $t\in[0,T]$ follows from the estimates in Lemma \ref{lem:HJ}.
\end{proof}

\section{Uniqueness of solutions}\label{sec:uniqueness}

As our main results, in this section we shall prove the uniqueness of solutions to the MFG system \eqref{eq:mfg} for $\beta\in\R$.

Suppose that $(u,\rho)$ is a solution pair of the mean field game system \eqref{eq:mfg} for some $\rho_0\in\sP_2(\mathbb R^d)$. Then $(t,x)\mapsto D_pH(x,-D_xu(t,x),\rho_t)$ is continuous in $t$ and globally Lipschitz continuous in the $x$-variable, the flow associated to the Fokker-Planck equation reads as
\begin{equation}\label{eq:flow}
X_t=\xi+\int_{0}^t D_pH(X_s,-D_xu(t,X_s),\rho_s)\dd s+\beta B_t, \quad t\in[0,T]
\end{equation}
where $\xi\in \mathbb L^2(\mathcal{F}_0;\rho_0)$.
When $\beta^2>0$, by standard $C^{2,\alpha}$ estimates for parabolic equations, we have that $u\in C^{1+\frac{\alpha}{2},2+\alpha}_{\rm loc}((0,T)\times\R^d)$. We define 
\begin{equation}\label{eq:Y}
Y_t:=-D_xu(t,X_t)
\end{equation}
and if $\beta^2>0$ we further define
\begin{equation}\label{eq:Z}
Z_t:=- D^2_{xx}u(t,X_t).
\end{equation}

We would like to justify that $(X,Y,Z)$ is a strong solution to the following forward-backward (stochastic) differential equation on $[0,T]$ associated to the MFG system \eqref{eq:mfg}
\begin{equation}\label{eq:FBSDE}
\left\{
\begin{array}{l}
\ds X_t=\xi+\int_{0}^tD_pH(X_s,Y_s,\rho_s)\dd s+\beta B_t, \\
\ds Y_t=-D_xg(X_T,\rho_T)+\int_t^T D_xH(X_s,Y_s,\rho_s)\dd s-\beta\int_t^TZ_s\dd B_s.
\end{array}
\right.
\end{equation}
Let us remark that in the case of $\b=0$, this system corresponds to a standard Hamiltonian system.

\begin{theorem}\label{thm:FBSDE}
Assume that \eqref{HLip} holds. Let $(u,\rho)$ be a solution pair of the mean field game system \eqref{eq:mfg} for some $\rho_0\in\sP_2(\mathbb R^d)$, and let $(X,Y,Z)$ be defined in \eqref{eq:flow}, \eqref{eq:Y} and \eqref{eq:Z}. Then $(X,Y,Z)$ is a strong solution to the forward-backward (stochastic) differential equation \eqref{eq:FBSDE}.
\end{theorem}
\begin{proof} 
The idea of the proof is based on the differentiation of the Hamilton-Jacobi-Bellman equation in the $x$-variable. To be able to justify this, we first regularize the equation.

Let $\{\zeta_n\}_n$ be a sequence of densities in $C_c^{\infty}(B_{\frac{1}{n}}(0))$. Define
\begin{equation*}\label{eq:un}
u_n(t,x):=\int_{\mathbb R^d}u(t,x-y)\zeta_n(y)\dd y, v_n(t,x):=D_xu_n(t,x) \text{ and } v(t,x)=D_xu(t,x).
\end{equation*}
Then
\begin{equation}\label{eq:HJn}
\partial_t u_n(t,x)=-\frac{\beta^2}{2}\Delta_x u_n(t,x)+\int_{\mathbb R^d}\zeta_n(y)H(x-y,-v(t,x-y),\rho_t)\dd y.
\end{equation}
We then differentiate \eqref{eq:HJn} in $x$ and obtain
\begin{eqnarray}\label{eq:DxHJn}
&&\partial_t v_n(t,x)+\frac{\beta^2}{2}\Delta_x v_n(t,x)\\
&=&\int_{\mathbb R^d}\zeta_n(y)[D_xH(x-y,-v(t,x-y),\rho_t)-D_pH(x-y,-v(t,x-y),\rho_t)D_xv(t,x-y)]\dd y.\nonumber
\end{eqnarray}
Let $(X,Y,Z)$ be defined in \eqref{eq:flow}, \eqref{eq:Y} and \eqref{eq:Z}. By It\^o formula and using \eqref{eq:DxHJn}, we have

\begin{align*}
&\dd v_n(t,X_t)\\
&=\left[\partial_tv_n(t,X_t)+\frac{\beta^2}{2}\Delta_xv_n(t,X_t)\right]\dd t+D_xv_n(t,X_t)\cdot \left[D_pH(X_t,-v(t,X_t),\rho_t)\dd t+\beta \dd B_t\right]\\
&=\left[\int_{\mathbb R^d}\zeta_n(y)[D_xH(X_t-y,-v(t,X_t-y),\rho_t)-D_pH(X_t-y,-v(t,X_t-y),\rho_t)D_xv(t,X_t-y)]\dd y\right]\dd t\\
&+D_xv_n(t,X_t)\cdot \left[D_pH(X_t,-v(t,X_t),\rho_t)\dd t+\beta \dd B_t\right].\\
\end{align*}\normalsize
Letting $n\to+\infty$ in the above equation, we have
\[
\dd v(t,X_t)= D_xH(x,-v(t,X_t),\rho_t)\dd t+\beta D_xv(t,X_t)\dd B_t,
\]
which is exactly the backward (stochastic) differential equation \eqref{eq:FBSDE}.
\end{proof}

\begin{theorem}\label{thm:mon_momentum}
We suppose that all the assumptions \eqref{hyp:H-L}-\eqref{hyp:L-mon-2} take place. Let us suppose that $(u^1,\rho^1)$ and $(u^2,\rho^2)$ are two solution pairs to \eqref{eq:mfg} with initial data $\rho^1_0,\rho^2_0$, respectively. Suppose that $X^i$, $i\in\{1,2\}$, stand for the flows of the vector fields $(t,x)\mapsto D_p H(t,x,-D_x u^i(t,x),\rho^i_t)$ defined in \eqref{eq:flow} with initial data $\xi^i\in \mathbb L^2(\mathcal{F}_0;\rho_0^i)$. Then $D_xu^1$ and $D_xu^2$ are jointly monotone along the flows $(X_t^1)_{t\in[0,T]}$ and $(X_t^2)_{t\in[0,T]}$, respectively. That is
\begin{equation}\label{eq:dispu}
\mathbb{E}\left([D_xu^1(t,X_t^1)-D_xu^2(t,X^2_t)]\cdot(X_t^1-X_t^2)\right)\ge 0,\ \forall\ t\in[0,T].
\end{equation}
\end{theorem}
\begin{proof}
Define $Y^i_t:=-D_xu^i(t,X_t^i)$ and $Z_t^i:=- D^2_{xx}u^i(t,X_t^i)$ for $i\in\{1,2\}$. Applying Theorem \ref{thm:FBSDE}, we have
\begin{equation}\label{BSDEi}
Y_t^i=-D_xg(X_T^i,\rho_T^i)+\int_t^TD_xH(X_s^i, Y_s^i,\rho_s^i)\dd s-\beta\int_t^TZ_s^i\dd B_s,
\end{equation}
and thus
\begin{align*}
&\dd\left[\left(D_xu^1(t,X_t^1)-D_xu^2(t,X_t^2)\right)\cdot(X_t^1-X_t^2) \right]\\
&= -\dd\left[\left(Y^1_t-Y^2_t\right)\cdot(X_t^1-X_t^2) \right]\\
&=(D_xH(X_t^1,Y_t^1,\rho_t^1)-D_xH(X_t^2,Y_t^2,\rho_t^2))\cdot (X^1_t-X^2_t)\dd t-\beta(Z_t^1-Z_t^2)\cdot(X_t^1-X_t^2)\dd B_t\\
&\quad-\left\{(Y_t^1-Y_t^2)\cdot\left(D_pH(X_t^1,Y_t^1,\rho_t^1)-D_pH(X_t^2,Y_t^2,\rho_t^2)\right)\right\}.
\end{align*}
Taking expectation on both sides above, we obtain
\begin{align}\label{eq:udisp}
&\dd\mathbb E\left\{\left[\left(D_xu^1(t,X_t^1)-D_xu^2(t,X_t^2)\right)\cdot(X_t^1-X_t^2) \right]\right\}\\
&=\mathbb E\Big\{(D_xH(X_t^1,Y_t^1,\rho_t^1)-D_xH(X_t^1,Y_t^1,\rho_t^1))\cdot (X^1_t-X^2_t)\Big\}\dd t\nonumber\\
&\quad-\mathbb E\Big\{\left[(Y_t^1-Y_t^2)\cdot\left(D_pH(X_t^1,Y_t^1,\rho_t^1)-D_pH(X_t^2,Y_t^2,\rho_t^2)\right)\right]\Big\}\nonumber\\
&\leq 0\nonumber
\end{align}
where in the last inequality we have used \eqref{Hdis}. Now, integrating the previous inequality in time on $[t,T]$, we find
\begin{align*}
&\mathbb{E}\Big\{(D_xu^1(t,X_t^1)-D_xu^2(t,X_t^2))\cdot(X^1_t-X^2_t)\Big\}\\
&\ge \mathbb{E}\left\{\left(D_xg(X^1_T,\rho^1_T)-D_xg(X^2_T,\rho^2_T)\right)\cdot(X^1_T-X^2_T)\right\}\ge 0.
\end{align*}
And so, the thesis of the theorem follows by \eqref{hyp:g_mon-2}.
\end{proof}

\begin{corollary}


If the assumptions in Theorem \ref{thm:mon_momentum} hold and $H(x,p,\mu)=\frac12|p|^2+f(x,\mu)$, then one has immediately 
\begin{equation}\label{eq:sharp}
W_2(\mathcal{L}_{X_t^1},\mathcal{L}_{X_t^2})\le W_2(\rho_0^1,\rho_0^2),\ \forall\ t\in[0,T].
\end{equation}
\end{corollary}
\begin{proof}Indeed, in this case one has 
\begin{align*}
&\frac{\dd}{\dd t}\frac12\mathbb{E} |X^1_t-X^2_t|^2 \\
&= \mathbb{E}\left\{(X^1_t-X^2_t)\cdot\left(D_pH(X^1_t,-D_xu^1(t,X^1_t),\rho^1_t)-D_pH(X^2_t,-D_xu^2(t,X^2_t),\rho^2_t)\right)\right\}\\
&=-\mathbb{E}\left\{(X^1_t-X^2_t)\cdot\left(D_xu^1(t,X^1_t)-D_xu^2(t,X^2_t)\right)\right\}\le 0,
\end{align*}
where in the last inequality we used the result from Theorem \ref{thm:mon_momentum}. Thus, the claim follows by integration over $[0,t]$ and choosing $\xi^i\in\mathbb L^2(\mathcal{F}_0;\rho_0^i)$, $i=1,2$, such that $W_2(\rho_0^1,\rho_0^2)=\left\{\mathbb E\left[|\xi^1-\xi^2|^2\right]\right\}^{1/2}$.
\end{proof}


\begin{remark}
For general Hamiltonians, we do not expect \eqref{eq:sharp} to hold true, since composition of monotone maps (in this case $D_pH$ and $D_xu$) in general fails to be monotone.
\end{remark}

\begin{theorem}\label{thm:unique-decay}
We suppose that all the assumptions \eqref{hyp:H-L}-\eqref{hyp:L-mon-2} take place. Let $(u^1,\rho^1)$ and $(u^2,\rho^2)$ be two solution pairs to \eqref{eq:mfg} with initial data $\rho^1_0,\rho^2_0\in\sP_2(\R^d)$, respectively. Then there exists $C>0$ depending only on $T$ and the data such that
\begin{align*}
\sup_{t\in[0,T]}W_2(\rho^1_t,\rho^2_t)\le C W_2(\rho^1_0,\rho^2_0),
\end{align*}
and
\begin{align*}
\sup_{t\in[0,T]}\left\|D_xu^1(t,\cdot)-D_xu^2(t,\cdot)\right\|_{L^\infty(\mathbb R^d)}\le CW_2(\rho_0^1,\rho_0^2).
\end{align*}
\end{theorem}

\begin{proof}
Let $(X^i, Y^i, Z^i)$, $i\in\{1,2\}$, be given as in Theorem \ref{thm:mon_momentum}. We note from \eqref{eq:udisp} 
\begin{align}\label{eq:ddisplnew}
&\frac{\dd}{\dd t}\mathbb E\left[\left(Y_t^1-Y_t^2\right)\cdot\left(X_t^1-X_t^2\right)\right]\\
&=-\mathbb E\left[\left(X_t^1-X_t^2\right)\cdot\left(D_xH(X_t^1,Y_t^1,\rho^1_t)-D_xH(X_t^2,Y_t^2,\rho_t^2)\right)\right]\nonumber\\
&\quad+\mathbb E\left[(D_pH(X_t^1,Y_t^1,\rho_t^1)-D_pH(X_t^2,Y_t^2,\rho_t^2))\cdot(Y_t^1-Y_t^2)\right].\nonumber
\end{align}
We integrate \eqref{eq:ddisplnew} from $0$ to $t$ and using \eqref{eq:dispu} one obtains
\begin{align*}
0\ge&\mathbb{E}\left[[Y^1_t-Y^2_t)]\cdot(X^1_t-X^2_t)\right]\\
=&\mathbb{E}\left[[Y^1_0-Y^2_0]\cdot(X^1_0-X^2_0)\right]\nonumber\\
&-\int_0^t\mathbb E\left[\left(X_s^1-X_s^2\right)\cdot\left(D_xH(X_s^1,Y_s^1,\rho_s^1)-D_xH(X_s^2,Y_s^2,\rho_s^2)\right)\right]\nonumber\\
&\qquad\,\,\,-\mathbb E\left[(D_pH(X_s^1,Y_s^1,\rho_s^1)-D_pH(X_s^2,Y_s^2,\rho_s^2))\cdot(Y_s^1-Y_s^2)\right]\dd s,\nonumber
\end{align*}
which by \eqref{HLip} implies that
\begin{align*}
&c_0\int_0^t\mathbb E\left[\left|Y_s^1-Y_s^2\right|^2\right]\dd s\\
\leq& -\mathbb{E}\left[[Y^1_0-Y^2_0]\cdot(X^1_0-X^2_0)\right]+C\int_0^t\mathbb E\left[\left|Y^1_s-Y^2_s\right|\left|X^1_s-X^2_s\right|\right]+\mathbb E\left[\left|X_s^1-X_s^2\right|^2\right]\dd s\nonumber.
\end{align*}
Applying Young's inequality, we derive
\begin{align*}
&\frac{c_0}{2}\int_0^t\mathbb E\left[\left|Y_s^1-Y_s^2\right|^2\right]\dd s\\
\leq& -\mathbb{E}\left[[Y^1_0-Y^2_0]\cdot(X^1_0-X^2_0)\right]+C\int_0^t\mathbb E\left[\left|X_s^1-X_s^2\right|^2\right]\dd s.
\end{align*}
Then
\begin{align*}
\mathbb{E}\left[|X^1_t-X^2_t|^2\right]&\leq\mathbb{E}\left[|X^1_0-X^2_0|^2\right]+\int_0^t\mathbb{E}\left[\left|D_pH(X^1_s,Y^1_s,\rho^1_s)- D_pH(X^2_s,Y^2_s,\rho^2_s)\right|^2\right]\dd s\\
&\le \mathbb{E}\left[|X^1_0-X^2_0|^2\right]+C\int_0^t\mathbb E\left[\left|Y_s^1-Y_s^2\right|^2\right]+\mathbb{E}\left[\left|X^1_s-X^2_s\right|^2\right]\dd s\\
&\le \mathbb{E}\left[|X^1_0-X^2_0|^2\right] -\mathbb{E}\left[[Y^1_0-Y^2_0]\cdot(X^1_0-X^2_0)\right]+C\int_0^t\mathbb E\left[\left|X_s^1-X_s^2\right|^2\right]\dd s.
\end{align*}
We recall that $Y_t^i=-D_xu^i(t,X^i_t)$ and note that $|D_x^2u^i|\leq C$ for $i=1,2$. We have
\begin{align*}
\mathbb{E}\left[|X^1_t-X^2_t|^2\right]\le& \mathbb{E}\left[|X^1_0-X^2_0|^2\right]+ C\mathbb E\left[\left|X_0^1-X_0^2\right|^2+\left|D_x u^1(0,X_0^1)-D_x u^2(0,X_0^2)\right|\left|X_0^1-X_0^2\right|\right]\\
&+C\int_0^t\mathbb E\left[\left|X_s^1-X_s^2\right|^2\right]\dd s\\
\le&C \mathbb{E}\left[|X^1_0-X^2_0|^2\right]+C\left\{\mathbb E\left[\left|D_xu^1(0,X_0^1)-D_xu^2(0,X_0^2)\right|^2\right]\right\}^{\frac{1}{2}}\left\{\mathbb E\left[\left|X_0^1-X_0^2\right|^2\right]\right\}^{\frac{1}{2}}\\
&+C\int_0^t\mathbb E\left[\left|X_s^1-X_s^2\right|^2\right]\dd s.
\end{align*}
Using Gr\"onwall's inequality and the fact that $W_2^2(\rho_t^1,\rho_t^2)\leq \mathbb{E}|X^1_t-X^2_t|^2$, we have
\small
\begin{align}\label{eq:rho}
W_2^2(\rho_t^1,\rho_t^2)&\leq \mathbb{E}\left[|X^1_t-X^2_t|^2\right]\\
\nonumber&\le C\left(\mathbb E\left[\left|X_0^1-X_0^2\right|^2\right]+\left\{\mathbb E\left[\left|D_xu^1(0,X_0^1)-D_xu^2(0,X_0^2)\right|^2\right]\right\}^{\frac{1}{2}}\left\{\mathbb E\left[\left|X_0^1-X_0^2\right|^2\right]\right\}^{\frac{1}{2}}\right).
\end{align}\normalsize
For any given $t_0\in [0,T]$, we now take the conditional expectation $\mathbb E\left[\cdot | X_{t_0}^i=x\right]$ on \eqref{BSDEi} for $i=1,2$: $\forall t\in [t_0,T]$
\begin{equation}\label{eq:Yix}
Y_t^{i,t_0,x}=-D_xg(X_T^{i,t_0,x},\rho_T^{i})+\int_t^TD_xH(X_s^{i,t_0,x}, Y_s^{i,t_0,x},\rho_s^{i})\dd s-\beta\int_t^TZ_s^{i,t_0,x}\dd B_s,
\end{equation}
where 
\[
X_t^{i,t_0,x}=x+\int_{t_0}^tD_pH(X_s^{i,t_0,x},-D_xu^i(t,X_s^{i,t_0,x}),\rho_s^i)\dd s+\beta B_t,
\]
$Y_t^{i,t_0x}=-D_xu^i(t,X_t^{i,t_0,x})$ and $Z_t^{i,t_0,x}=- D^2_{xx}u^i(t,X_t^{i,t_0,x})$. 
By \eqref{HLip} and by the global Lipschitz property of $x\mapsto D_xu^i(t,x)$, uniformly on $t\in[0,T]$, it follows from standard SDE arguments that 
\begin{equation}\label{eq:DeltaXt0x}
\Big(\mathbb E\Big[\sup_{s\in [t_0,T]}\big|X_s^{1,t_0,x}-X_s^{2,t_0,x}\big|^2\Big]\Big)^{\frac{1}{2}}\leq C\int_{t_0}^T\big[\|D_xu^1(s,\cdot)-D_xu^2(s,\cdot)\|_{L^\infty(\mathbb R^d)}+W_2(\rho_s^1,\rho_s^2)\big]ds.
\end{equation}
Letting $t=t_0$ and taking expectation on \eqref{eq:Yix}, we have
\begin{align*}
D_xu^i(t_0,x)=\mathbb E\left[D_xg(X_T^{i,t_0,x},\rho_T^i)\right]-\int_{t_0}^T\mathbb E\left[D_xH(X_s^{i,t_0,x},Y_s^{i,t_0,x},\rho_s^i)\right]\dd s,\quad \text{for $i=1,2$,}
\end{align*}
and thus by \eqref{hyp:D_xg}, \eqref{HLip} and \eqref{eq:DeltaXt0x} we have
{\small
\begin{align*}
&\quad|D_xu^1(t_0,x)-D_xu^2(t_0,x) |\\
&\le C\left[\sup_{s\in[t_0,T]}W_2(\rho^1_s,\rho^2_s)+\int_{t_0}^T\mathbb E\left[\left|D_xu^1(s,X_s^{1,t_0,x})-D_xu^2(s,X_s^{2,t_0,x})\right|\right]\dd s\right]\\
&+C\left[\mathbb E |X_T^{1,t_0,x}-X_T^{2,t_0,x}|+\int_{t_0}^T\mathbb E\left|X_s^{1,t_0,x}-X_s^{2,t_0,x}\right|\dd s\right]\\
&\le C\left[\sup_{s\in[t_0,T]}W_2(\rho^1_s,\rho^2_s)+\int_{t_0}^T\left\|D_xu^1(s,\cdot)-D_xu^2(s,\cdot)\right\|_{L^\infty(\mathbb R^d)}\dd s\right]\\
&+C\left[\left(\mathbb E\big[ |X_T^{1,t_0,x}-X_T^{2,t_0,x}|^2\big]\right)^\frac12+\int_{t_0}^T\left(\mathbb E\big[|X_s^{1,t_0,x}-X_s^{2,t_0,x}|^2\big]\right)^\frac12\dd s\right]\\
&\le C\left[\sup_{s\in[t_0,T]}W_2(\rho^1_s,\rho^2_s)+\int_{t_0}^T\left\|D_xu^1(s,\cdot)-D_xu^2(s,\cdot)\right\|_{L^\infty(\mathbb R^d)}\dd s\right].
\end{align*}}
By Gronwall's inequality, we derive
\begin{align}\label{eq:Du}
\left\|D_xu^1-D_xu^2\right\|_{L^\infty([0,T]\times\mathbb R^d)}\le C\sup_{t\in[0,T]}W_2(\rho^1_t,\rho^2_t).
\end{align}
Plugging \eqref{eq:Du} into \eqref{eq:rho} and applying Young's inequality, we obtain
\begin{align*}
\sup_{t\in[0,T]}W_2(\rho_t^1,\rho_t^2)\leq C\left\{\mathbb E\left[\left|X^1_0-X^2_0\right|^2\right]\right\}^{\frac{1}{2}}.
\end{align*}
We can choose $\xi^i$ to be such that $W_2(\rho_0^1,\rho_0^2)=\left\{\mathbb E\left[\left|\xi^1-\xi^2\right|^2\right]\right\}^{\frac{1}{2}}$ and thus
\begin{align}\label{eq:rhonew}
\sup_{t\in[0,T]}W_2(\rho_t^1,\rho_t^2)\leq CW_2(\rho_0^1,\rho_0^2).
\end{align}
Combining \eqref{eq:Du} and \eqref{eq:rhonew}
\begin{align}\label{eq:Dunew}
\left\|D_xu^1-D_xu^2\right\|_{L^\infty([0,T]\times\mathbb R^d)}\le CW_2(\rho_0^1,\rho_0^2).
\end{align}

\end{proof}

\begin{corollary}\label{cor:uniqueness}
We suppose that all the assumptions \eqref{hyp:H-L}-\eqref{hyp:L-mon-2} take place. The mean field game system \eqref{eq:mfg} admits at most one solution pair $(u,\rho)$.
\end{corollary}
\begin{proof}
It is immediate that Theorem \ref{thm:unique-decay} yields the uniqueness of $\rho$ and $D_x u$. By the uniqueness of solutions to the HJB component of \eqref{eq:mfg} (given $\rho$), the uniqueness of $u$ follows.
\end{proof}

{\sc Acknowledgements.} ARM acknowledges the support of the Heilbronn Institute for Mathematical Research and the UKRI/EPSRC Additional Funding Programme for Mathematical Sciences through the focused research grant ``The master equation in Mean Field Games''. ARM has also been partially supported by the EPSRC New Investigator Award ``Mean Field Games and Master equations'' under award no. EP/X020320/1 and by the King Abdullah University of Science and Technology Research Funding (KRF) under award no. ORA-2021-CRG10-4674.2. CM gratefully acknowledges the support by CityU Start-up Grant 7200684, Hong Kong RGC Grant ECS 21302521 and Hong Kong RGC Grant GRF 11311422.

\end{document}